\documentclass[11pt,a4paper]{amsart}
\usepackage{amsthm}

\usepackage[utf8]{inputenc}
\usepackage{bbm,amssymb,mathtools,mathrsfs}

\usepackage[usenames,dvipsnames]{xcolor}

\usepackage[colorlinks,linkcolor=BrickRed,citecolor=OliveGreen,urlcolor=black,hypertexnames=true]{hyperref}

\usepackage[margin=3.5cm]{geometry}

\usepackage{enumerate}

\makeatletter
\def\l@subsection{\@tocline{2}{-4pt}{2.5pc}{5pc}{}}
\renewcommand\tocchapter[3]{%
  \indentlabel{\@ifnotempty{#2}{\ignorespaces#2.\quad}}#3%
}
\newcommand\@dotsep{4.5}
\def\@tocline#1#2#3#4#5#6#7{\relax
  \ifnum #1>\c@tocdepth 
  \else
    \par \addpenalty\@secpenalty\addvspace{#2}%
    \begingroup \hyphenpenalty\@M
    \@ifempty{#4}{%
      \@tempdima\csname r@tocindent\number#1\endcsname\relax
    }{%
      \@tempdima#4\relax
    }%
    \parindent\z@ \leftskip#3\relax \advance\leftskip\@tempdima\relax
    \rightskip\@pnumwidth plus1em \parfillskip-\@pnumwidth
    #5\leavevmode\hskip-\@tempdima{#6}\nobreak
    \leaders\hbox{$\m@th\mkern \@dotsep mu\hbox{.}\mkern \@dotsep mu$}\hfill
    \nobreak
    \hbox to\@pnumwidth{\@tocpagenum{#7}}\par
    \nobreak
    \endgroup
  \fi}
\makeatother
\AtBeginDocument{%
\makeatletter
\expandafter\renewcommand\csname r@tocindent0\endcsname{0pt}
\makeatother
}
\def\l@subsection{\@tocline{2}{0pt}{2.5pc}{5pc}{}}

\input xy
\xyoption{all}

\DeclareMathOperator{\id}{id}
\DeclareMathOperator{\Sym}{Sym}

\DeclareMathOperator{\tr}{tr}
\DeclareMathOperator{\sgn}{sgn}
\DeclareMathOperator{\Tr}{Tr}

\DeclareMathOperator{\End}{End}

\DeclareMathOperator{\Herm}{Herm}

\DeclareMathOperator{\sign}{sign}

\DeclareMathOperator{\Nil}{Nil}
\DeclareMathOperator{\Hom}{Hom}

\DeclareMathOperator{\Int}{Int}
\DeclareMathOperator{\rk}{rank}

\DeclareMathOperator{\Supp}{Supp}

\newcommand{\N}{\mathbb{N}}
\newcommand{\Z}{\mathbb{Z}}

\newcommand{\CM}{\mathscr{M}}

\newcommand{\CU}{\mathscr{U}}
\newcommand{\CF}{\mathscr{F}}
\newcommand{\CG}{\mathscr{G}}
\newcommand{\CS}{\mathscr{S}}
\newcommand{\CE}{\mathscr{E}}

\newcommand{\Sper}{\operatorname{Sper}}
\newcommand{\Spec}{\operatorname{Spec}}

\DeclareMathOperator{\Sperm}{Sper^{\mathrm{max}}}

\newcommand{\fm}{\mathfrak{m}}
\newcommand{\fp}{\mathfrak{p}}

\newcommand{\s}{\sigma}

\newcommand{\Sign}{\operatorname{Sign}}
\newcommand{\ox}{\otimes}
\newcommand{\x}{\times}

\newcommand{\Qf}{\mathrm{Frac}}

\newcommand{\ve}{\varepsilon}

\newcommand{\ovl}{\overline}

\newcommand{\qf}[1]{\langle #1\rangle}
\newcommand{\Pf}[1]{\langle\!\langle #1\rangle\!\rangle}

\newcommand{\sm}{\setminus}

\newcommand{\vf}{\varphi}
\newcommand{\vt}{\vartheta}

\newcommand{\ns}{\mathrm{ns}}

\newcommand{\op}{\mathrm{op}}

\newcommand{\ad}{\mathrm{ad}}

\newtheorem{thm}{Theorem}[section]

\newtheorem{lemma}[thm]{Lemma}

\newtheorem{prop}[thm]{Proposition}

\theoremstyle{definition}
\newtheorem{defn}[thm]{Definition}

\newtheorem*{exa*}{Example}

\theoremstyle{remark}
\newtheorem{rem}[thm]{Remark}

\numberwithin{equation}{section}

\allowdisplaybreaks

\begin{document}
	
\title{A Knebusch trace formula for Azumaya algebras with involution}

\author[V. Astier]{Vincent Astier}
\author[T. Unger]{Thomas Unger}
\address{School of Mathematics and Statistics, 
University College Dublin, Belfield, Dublin~4, Ireland}
\email{vincent.astier@ucd.ie}
\email{thomas.unger@ucd.ie}



\begin{abstract}
  We establish a trace formula for signatures of hermitian forms
  over Azumaya algebras with involution, extending
  Knebusch's work on symmetric bilinear forms over finite \'etale 
  extensions of commutative base rings. 
  As an application when the base ring is semilocal, we obtain an exact 
  sequence for total signatures, related to Pfister's local-global principle
  and the notion of stability index. 
\end{abstract}


\maketitle


\parskip\smallskipamount
\tableofcontents
\parskip0pt


\section{Introduction}

Let $R$ be a commutative ring. In the mid-1970s, Knebusch proved a trace formula
for signatures of symmetric bilinear forms over finite \'etale extensions, and
more generally Frobenius extensions, of $R$, cf. \cite{kne75, kne76, kne77}.
Knebusch's approach predates the discovery of the real spectrum $\Sper R$ by
Coste and Roy, cf. \cite[Chapter~7]{BCR} and \cite{knebusch84}, and the
signatures used are ring homomorphisms from the Witt ring $W(R)$ to $\Z$.

In this paper we are interested in a trace formula for hermitian forms,
expressed in terms of signatures at orderings $\alpha\in \Sper R$. Specifically,
after the required prerequisites in Section~\ref{sec2}, we establish such a
formula for hermitian forms over Azumaya algebras with involution over
commutative rings in Section~\ref{sec:ktf}. We proved a similar result for
central simple algebras with involution over fields in our earlier paper
\cite{A-U-Kneb}, which is also the starting point of our development of real
algebra for ``well-behaved'' rings with involution.

The main thread of our approach is signatures of hermitian forms. The behaviour
of signatures is generally pleasing in terms of Witt groups and the Harrison
topology on $\Sper R$, cf. \cite{A-U-Az-cont}. The primary mechanism for their
definition is hermitian Morita theory, and we provide a technical appendix on
this topic. As an application of the Knebusch trace formula, we prove the
existence of a nonsingular reference form of constant signature $2^m$ in
Section~\ref{sec4}, while in Section~\ref{sec5} we briefly dwell on an exact
sequence for total signatures when the base ring is semilocal (and which is
connected to the stability index in the commutative case).

\section{Preliminaries}\label{sec2}

In this paper all rings are assumed unital and associative with $2$ invertible,
and all fields are assumed to have characteristic different from $2$.

\subsection{Hermitian forms}

Let $(A,\s)$ be a ring with involution. A hermitian module over $(A,\s)$ is a
pair $(M,h)$ where $M$ is a finitely generated projective right $A$-module and
$h\colon M\x M \to A$ is a hermitian form with respect to $\s$. We often refer
to $(M,h)$ as a hermitian form and do not always indicate what $M$ is. We denote
isometry of forms by $\simeq$. For $a_1,\ldots, a_\ell \in \Sym(A,\s):=\{a\in A
\mid \s(a)=a\}$ we denote by $\qf{a_1,\ldots, a_\ell}_\s$ the diagonal hermitian
form 
\[ 
  A^\ell\x A^\ell \to A,\ (x,y)\mapsto \sum_{i=1}^\ell \s(x_i)a_i y_i. 
\]
Note that not every hermitian form is isometric to a diagonal hermitian form.

\begin{defn}\label{direct-product}
  For $i=1, \ldots, t$,
  let $(A_i, \s_i)$ be rings with involution and let $(M_i, h_i)$ be a hermitian
  form over $(A_i, \s_i)$. Let $(A,\s) := (A_1, \s_1) \x
  \cdots \x (A_t, \s_t)$. We denote by $(M_1, h_1) \x \cdots \x (M_t, h_t)$ the
  hermitian form $(M,h)$ over $(A,\s)$ defined by
  \begin{itemize}
    \item $M = M_1 \x \cdots \x M_t$, with the obvious structure of right
      $A$-module;
    \item $h((m_1, \ldots, m_t), (m'_1, \ldots, m'_t)) := (h_1(m_1, m'_1),
      \ldots, h_t(m_t,m'_t))$
  \end{itemize}
  (it is straightforward to check that $(M,h)$ is a hermitian form over $(A,\s)$).
\end{defn}

\begin{prop}\label{dp}
  With the same notation as in Definition~\ref{direct-product}, assume that
  there are commutative rings $R_i$ such that $A_i$ is a finitely generated
  $R_i$-module, for $i=1, \ldots, t$. Then
  \begin{enumerate}[{\quad\rm (1)}]
    \item $h$ is nonsingular if and only if every $h_i$ is nonsingular;
    \item if every $h_i$ is hyperbolic, then $h$ is hyperbolic.
  \end{enumerate}
\end{prop}
\begin{proof}
  Let $R := R_1 \x \cdots \x R_t$. Then $A$ is a finitely generated $R$-module.

  The first statement is a consequence of the fact that $h$ is nonsingular if
  and only if $h \ox_R R/\fm$ is nonsingular, for every  maximal ideal $\fm$ of
  $R$, cf. \cite[Chapter~I, Lemma~7.1.3]{knus91}.

  We prove the second statement. We can write $M_i = L_i \oplus P_i$ with
  $h_i(L_i,L_i) = 0$ and $h_i(P_i,P_i) = 0$ (cf. \cite[Section~2.2]{first23}).
  So $M = \prod_{i=1}^t (L_i \oplus P_i) \cong (\prod_{i=1}^t L_i) \oplus
  (\prod_{i=1}^t P_i)$. Let $L = \prod_{i=1}^t L_i$ and $P = \prod_{i=1}^t P_i$.
  We check that $h(L,L) = 0$; the proof of $h(P,P) = 0$ is similar. It suffices
  to show that $h(\ell_i,\ell'_j) = 0$ for each $\ell_i \in L_i$ and $\ell'_j
  \in L_j$. If $i \not = j$ then $h(\ell_i,\ell'_j) = 0$ and if $i = j$ then
  $h(\ell_i,\ell'_i) = h_i(\ell_i,\ell'_i) = 0$. 
\end{proof}

\subsection{Finite \'etale algebras}

Let $R$ be a commutative ring. We say that an $R$-algebra is \emph{\'etale} over
$R$ if it is a commutative, separable, flat, finitely presented $R$-algebra, and
\emph{finite \'etale} if it is \'etale and a finitely generated $R$-module.

We have the following sequence of
implications (in fact, equivalences). These are well-known, but we could not
find them in a convenient single reference.
\begin{align*}
    \text{$T$ is a } &\text{finite \'etale $R$-algebra} \\
      &\Rightarrow \text{$T$ is a separable flat $R$-algebra and finitely 
      presented as
      $R$-module} \\
      &\qquad \text{(by \cite[first part of Remark~2.4]{SW21})} \\
      &\Rightarrow \text{$T$ is a separable $R$-algebra and a finitely generated
      projective $R$-module} \\
      &\qquad \text{(by \cite[second part of Remark~2.4]{SW21})} \\
      &\Rightarrow \text{$T$ is a separable, flat, finitely presented
      $R$-algebra and is finitely} \\
      &\qquad \text{generated as $R$-module} \\
      &\qquad \text{(since projective modules are flat and by \cite[Exercise
      9.2.2(2)]{ford17})} \\
      &\Rightarrow \text{$T$ is a finite \'etale $R$-algebra},
  \end{align*}
so that all these statements are equivalent.

\subsection{Azumaya algebras with involution}

Let $R$ be a commutative ring. An
$R$-algebra $A$ is an \emph{Azumaya $R$-algebra} if $A$ is a faithful finitely
generated projective $R$-module and the map 
\[ 
  A\ox_R  A^{\op}\to \End_R(A),\ a\ox b^\op\mapsto [x\mapsto axb]
\] 
is an isomorphism of $R$-algebras (here $A^\op$ denotes the \emph{opposite
algebra} of $A$, which coincides with $A$ as an $R$-module, but with twisted
multiplication $a^\op b^\op=(ba)^\op$). The centre $Z(A)$ is equal to $R$. See 
\cite[Chapter~III, (5.1)]{knus91}, for example. We recall the following results:

\begin{prop}[{\cite[Lemma~2.1 and Theorem~2.2]{Saltman99}}]\label{basic-Az}\mbox{}
  \begin{enumerate}[{\quad\rm (1)}]
    \item If $K$ is a field, then $A$ is an Azumaya $K$-algebra if and only
      if $A$ is a central simple $K$-algebra.
    \item If $A$ is an Azumaya $R$-algebra and $f: R \rightarrow S$ is a
    morphism of commutative rings, then $A \ox_R S$ is an Azumaya $S$-algebra.
  \end{enumerate}
\end{prop}

\begin{defn}[{\cite[Section~1.4]{first23}}]\label{first-def}  
  We say that $(A,\s)$ is an \emph{Azumaya
  algebra with involution over $R$} if the following conditions hold:
  \begin{itemize}
    \item $A$ is an $R$-algebra with $R$-linear involution $\s$;
    \item $A$ is separable projective over $R$;
    \item the homomorphism $R \rightarrow A$, $r \mapsto r \cdot 1_A$ identifies
      $R$ with the set $\Sym(Z(A),\s)$ of $\s$-symmetric elements of $Z(A)$.
  \end{itemize}
\end{defn}

This definition is motivated by:

\begin{prop}[{\cite[Proposition~1.1]{first23}}]\label{prop:Az}
  Let $A$ be an $R$-algebra. Then
  $A$ is Azumaya over $Z(A)$ and $Z(A)$ is finite \'etale over $R$ if and only
  if $A$ is projective as an $R$-module and separable as an $R$-algebra.
\end{prop}

A further link with Azumaya $R$-algebras is provided by
\cite[Lemma~2.5]{A-U-Az-PLG}:
\begin{lemma}\label{twodefs}
  Let $A$ be an $R$-algebra with $R$-linear involution $\s$ such that $A$ is an
  Azumaya algebra over $Z(A)$, $Z(A)$ is $R$ or a quadratic \'etale extension of
  $R$, and $R = \Sym(Z(A),\s)$. Then $(A,\s)$ is an Azumaya algebra with
  involution over $R$.
  
  The converse holds if $R$ is connected.
\end{lemma}

\noindent\textbf{Assumption for the remainder of the paper:} $R$ is a
commutative unital ring with $2$ invertible (in particular, we identify
quadratic and symmetric bilinear forms over $R$), and $(A,\s)$ is an Azumaya
algebra with involution over $R$. 
\medskip
 
Note that $A$ is Azumaya over $Z(A)$ by Proposition~\ref{prop:Az}, but may not
be Azumaya over $R$.

We recall the following, proved in \cite[Section~1.4, second 
paragraph]{first23}:
\begin{prop}[Change of base ring]\label{change-of-base}
  Let $T$ be a commutative $R$-algebra. Then  $Z(A \ox_R T) = Z(A) \ox_R T$
  and $(A \ox_R T, \s \ox \id)$ is an 
  Azumaya algebra with
  involution over $T$ (and in particular $Z(A \ox_R T) \cap
  \Sym(A \ox_R T, \s \ox \id) = T$).
\end{prop}

\begin{rem}
  If $R=F$ is actually a field, then $(A,\s)$ is an Azumaya algebra with
  involution over $R$ if and only if it is a central simple $F$-algebra with
  involution in the sense of \cite[Sections~2.A, 2.B]{BOI} (i.e., we allow for
  the possibility that $Z(A)=F\x F$). We call such algebras \emph{$F$-algebras
  with involution}.
\end{rem}

Recall that if $(B,\tau)$ is an $F$-algebra with involution, then $\tau$ is said
to be of the first kind if $Z(B) \subseteq \Sym(B,\tau)$, and of the second kind
(or of unitary type) otherwise. Involutions of the first kind are further
divided into those of orthogonal type and those of symplectic type, depending on
the dimension of $\Sym(B,\tau)$, cf. \cite[Sections~2.A and 2.B]{BOI}.

\begin{defn}[See {\cite[Section~1.4]{first23}}]\label{def-type}
  We say that the involution $\s$ on $A$ 
  is \emph{of orthogonal, symplectic or unitary type at 
  $\fp\in \Spec R$} if 
  $(A \ox_R \kappa(\fp), \s \ox \id)$ is a $\kappa(\fp)$-algebra with 
  involution of orthogonal, symplectic or unitary type, respectively,  
  where $\kappa(\fp)$ denotes the
  field of fractions $\Qf(R/\fp)$. We also say that
  $\s$ is \emph{of orthogonal, symplectic or unitary type} 
  if it is  of orthogonal, symplectic or unitary type, respectively, 
   at all $\fp\in \Spec R$.
\end{defn}

\begin{prop}[See {\cite[Proposition~1.21]{first23}}  for a more detailed 
  statement]\label{rem:quad-et}
  If $R$ is connected, then precisely 
  one of the following  holds:
  \begin{enumerate}[{\quad\rm (1)}]
    \item $\s$ is of orthogonal type, 
    $Z(A)=R$, and $\s|_{Z(A)}=\id_{Z(A)}$;
    \item $\s$ is of symplectic type, 
    $Z(A)=R$, and $\s|_{Z(A)}=\id_{Z(A)}$;
    \item $\s$ is of unitary type, $Z(A)$ is a quadratic \'etale $R$-algebra,
      and $\s|_{Z(A)}$ is the standard involution on $Z(A)$
      (i.e., the unique involution $\vt$ on $Z(A)$ such that $R = 
      \Sym(Z(A), \vt)$, cf. \cite[Chapter~I, Section~1.3 and 
      Proposition~1.3.4]{knus91}).
  \end{enumerate}
\end{prop}

\begin{prop}\label{prop:same_type}
  Let $\vf : R \rightarrow T$ be a morphism  of commutative rings, and let
  $\fp \in \Spec T$. Then the involutions $\s \ox \id$ on $A \ox_R
  \kappa(\vf^{-1}(\fp))$ and $\s \ox \id \ox \id$ on $A \ox_R T \ox_T
  \kappa(\fp)$ are of the same type.

  In particular, if $R$ is connected, then the involution $\s \ox \id$ on $A
  \ox_R T$ is of the same type as $\s$.
\end{prop}

\begin{proof}
  We first recall that the result holds in the field case: if $R$ and $T$ are
  fields, then the involutions $\s$ on $A$ and $\s \ox \id$ on $A \ox_R T$ 
  are of the same type. 

  The map $\vf$ induces an injective morphism of rings from $R/\vf^{-1}(\fp)$ to
  $T/\fp$, and thus a morphism of fields $\bar \vf: \kappa(\vf^{-1}(\fp))
  \rightarrow \kappa(\fp)$. Clearly:
  \[\xymatrix@R=0.5ex{
    T \ox_T \kappa(\fp) \ar[r]^-{\sim} & \kappa(\fp) \ar[r]^-{\sim} &
    \kappa(\vf^{-1}(\fp)) \ox_{\bar \vf} \kappa(\fp) \\
    t \ox x \ar@{|-{>}}[r] & \bar t x \ar@{|-{>}}[r] & 1 \ox \bar t x,
  }\]
  which induces an isomorphism from $A \ox_R T \ox_T \ox \kappa(\fp)$ to $A
  \ox_R \kappa(\vf^{-1}(\fp)) \ox_{\bar \vf} \kappa(\fp)$, and thus an
  isomorphism of algebras with involution
  \[(A \ox_R T \ox_T \kappa(\fp), \s \ox \id \ox \id) \rightarrow 
  (A \ox_R \kappa(\vf^{-1}(\fp)) \ox_{\bar \vf} \kappa(\fp), \s \ox \id \ox
  \id).\]
  Therefore it suffices to show that the involutions $\s \ox \id$ on $A \ox_R
  \kappa(\vf^{-1}(\fp))$ and $\s \ox \id \ox \id$ on $A \ox_R
  \kappa(\vf^{-1}(\fp)) \ox_{\bar \vf} \kappa(\fp)$ are of the same type, which
  follows from the field case recalled above.

  We now prove the second statement of the Proposition. Let $\fp \in \Spec T$.
  Observe that, by Definition~\ref{def-type}, the involutions $\s$ on $A$ and
  $\s \ox_R \id$ on $A \ox_R \kappa(\vf^{-1}(\fp))$ are of the same type. By the
  first part of the Proposition, $\s \ox_R \id$ on $A \ox_R
  \kappa(\vf^{-1}(\fp))$ and $\s \ox \id \ox \id$ on $A \ox_R T \ox_T
  \kappa(\fp)$ are also of the same type. Therefore, the type of $\s \ox \id \ox
  \id$ on $A \ox_R T$ is constant on $\Spec T$ and equal to the type of $\s$.
\end{proof}

\subsection{The trace map}\label{inv-trace}

Let $T$ be a finite \'etale $R$-algebra. We denote by $\Tr_{T/R} \in \Hom(T,R)$
the trace map from $T$ to $R$ afforded by $T$, in the terminology of
\cite[p.~91]{DMI71} (this is the reference used in \cite[\S3]{kne75}) or
\cite[Definition~4.6.5]{ford17}. Note that if $T$ is finitely generated and
free, then $\Tr_{T/R}$ is the usual trace, defined via the left regular
representation of $T$, cf. \cite[Exercise 4.6.14]{ford17}. Furthermore, the
associated symmetric bilinear form $\tr_{T/R} : T \times T \rightarrow R$,
$(x,y) \mapsto \Tr_{T/R}(xy)$ is nonsingular. Indeed, it follows from the proof
of \cite[Corollary~4.6.8, end of first paragraph]{ford17} that the adjoint
linear map of $\tr_{T/R}$ is bijective.

We record the following known result:
\begin{lemma}\label{trace-extension}
  Let $S$ be a commutative $R$-algebra. Then $\Tr_{T \ox_R
  S/S} = \Tr_{T/R} \ox_R \id_S$.
\end{lemma}
\begin{proof}
  As observed in \cite[p.~153, proof of (3) implies (1)]{ford17}, 
  a projective dual
  basis of $T$ over $R$ gives a projective dual basis of $T \ox_R S$ over $S$,
  and the result follows. 
\end{proof}

We define
\[
  \Tr_{A\ox_R T/A}:=\id_A \ox_R \Tr_{T/R}:A\ox_R T \to A,
\]
the map induced by the trace map $\Tr_{T/R}$. Observe that it is an involution
trace in the sense of \cite[Chapter~I, Proposition~7.2.4 and following
paragraph]{knus91}: Checking the first two points of the definition is direct,
while the third one follows from the fact that the form that needs to be
nonsingular is actually $\qf{1}_\s \ox_R \tr_{T/R}$, which is nonsingular, as a
tensor product of two nonsingular forms.

It follows that if $h$ is hyperbolic, or nonsingular, then $\Tr_{A \ox_R
T/A}^*(h)$ is respectively hyperbolic (cf.
\cite[Chapter~I, bottom of p.~40]{knus91}), or nonsingular (cf.
\cite[Chapter~I, 10.3, top of p.~62]{knus91}).

\subsection{Real algebra}\label{secra}

The main references for this section are \cite{knebusch84} and \cite{KSU}.
\smallskip

{\bf The real spectrum:} We denote the real spectrum of $R$ (the set of all
orderings on $R$) by $\Sper R$ and by
\[
 \mathring H(r) := 
 \{\alpha \in \Sper R \mid r > 0 \text{ at } \alpha\} 
    = \{\alpha \in \Sper R \mid r \in \alpha \setminus -\alpha\}
\]
(for $r \in R$) the sets that form the standard subbasis of open sets of the
Harrison topology on $\Sper R$. The space $\Sper R$ is compact (by which we mean
quasi-compact) for this topology, cf. \cite[Corollary~3.4.6]{KSU}.

If $\vf : R \rightarrow T$ is a morphism of commutative rings, $\alpha \in \Sper
R$, and $\gamma \in \Sper T$, we say that $\gamma$ is an extension of $\alpha$
if $\alpha = \vf^{-1}(\gamma)$, and we denote by $\Sper T/\alpha$ the set of all
extensions of $\alpha$ in $\Sper T$.

Let $\alpha \in\Sper R$. The support of $\alpha$ is the prime ideal
$\Supp(\alpha) := \alpha \cap -\alpha \in \Spec R$, and  
we define the following rings and maps:
\[\xymatrix{
  R \ar[rr]^-{\rho_\alpha} \ar@/_2pc/[rrr]_{r_\alpha} & & \kappa(\alpha):=\Qf
  (R/\Supp(\alpha)) \ar[r] & k(\alpha)
},\]
where
\begin{itemize}
  \item $\rho_\alpha$ is the canonical map
    from $R$ to $\kappa(\alpha)$,
  \item $k(\alpha)$ is a real closure of $\kappa(\alpha)$ at $\bar\alpha$, the
    ordering induced by $\alpha$ on $\kappa(\alpha)$,
  \item and $r_\alpha$ is the composition of the two horizontal maps.
\end{itemize}

We denote the unique ordering on $k(\alpha)$ by $\tilde\alpha$, and also define
\[
(A(\alpha), \sigma(\alpha)) := (A \ox_R \kappa(\alpha), \s \ox \id).
\]

\begin{rem}\label{rk-ms}
  By Proposition~\ref{change-of-base}, $(A \ox_R k(\alpha), \s \ox \id)$ is an Azumaya algebra with
  involution over the field $k(\alpha)$, and by Lemma~\ref{twodefs} 
  the centre of $A \ox_R k(\alpha)$
  is equal to $k(\alpha)$ or a quadratic \'etale extension of
  $k(\alpha)$. Therefore, $Z(A \ox_R k(\alpha))$ is equal to $k(\alpha)$,
  $k(\alpha)(\sqrt{-1})$ or $k(\alpha) \x k(\alpha)$. In the first two cases, $A
  \ox_R k(\alpha)$ is  simple, and thus 
  \[A \ox_R k(\alpha) \cong M_{n_\alpha}(D_\alpha), \text{ with } D_\alpha \in
  \{k(\alpha), k(\alpha)(\sqrt{-1}), (-1,-1)_{k(\alpha)}\},\]
  for some $n_\alpha \in \N$.
  In the third case, $A \ox_R k(\alpha)$ is the product of two simple
  algebras with centre $k(\alpha)$, and we have
  \[A \ox_R k(\alpha) \cong M_{n_\alpha}(D_\alpha), \text{ with } D_\alpha \in
  \{k(\alpha) \x k(\alpha), (-1,-1)_{k(\alpha)} \x (-1,-1)_{k(\alpha)}\},\]
  for some $n_\alpha \in \N$.
  Note that
  \[\dim_{k(\alpha)} A \ox_R k(\alpha) \in \{n_\alpha^2, 2n_\alpha^2,
  4n_\alpha^2, 8n_\alpha^2\}.\]
\end{rem}

{\bf The maximal real spectrum:}
We denote by $\Sperm R$ the set of all elements of $\Sper R$ that are maximal
for inclusion, equipped with the topology induced by $\Sper R$. For $\alpha \in
\Sper R$, the set $\{\beta \in \Sper R \mid \alpha \subseteq \beta\}$ is totally
ordered by inclusion (cf. \cite[Theorem~3.6.7]{KSU}), and every $\alpha \in \Sper R$ is included in a unique
element $\alpha^*$ of $\Sperm R$, cf. \cite[Corollary~3.6.8]{KSU}. The map \[\Sper R \rightarrow \Sperm R,\
\alpha \mapsto \alpha^*\] is continuous, cf. \cite[Proposition~3.6.16]{KSU}.
\medskip

{\bf Product of rings:}
If $R = R_1 \x \cdots \x R_n$ is a product, it is well-known that we can identify each $\Spec R_i$ with a subset of $\Spec R$ via
\[\Spec R_i \rightarrow \Spec R,\quad \fp \mapsto R_1 \x \cdots \x R_{i-1} \x \fp \x R_{i+1} \x \cdots \x R_n.\]
Under this identification, we have $\Spec R = \Spec R_1 \dot \cup \cdots \dot \cup \Spec R_n$.
Similarly, working with the support of the orderings, we have
\begin{equation}\label{union}
  \Sper R = \Sper R_1 \dot \cup \cdots \dot \cup \Sper R_n,
\end{equation}
where
\[\Sper R_i \rightarrow \Sper R,\quad \alpha \mapsto R_1 \x \cdots \x R_{i-1} \x \alpha \x R_{i+1} \x \cdots \x R_n.\]

{\bf Signatures of $\boldsymbol{R}$:}
The following content, about signatures of rings, can be found in \cite[Section
5 up to p.~89]{knebusch84}.

A signature of $R$ is a morphism of rings from $W(R)$, the
Witt ring of nonsingular symmetric bilinear forms over $R$, to $\Z$. 
If $\alpha \in \Sper R$, the (Sylvester) signature of symmetric bilinear forms at $\alpha$ defines a
morphism of rings $\sign_\alpha : W(R) \rightarrow \Z$, which is then a signature of
$R$.
The set $\Sign R$ of all signatures
of $R$ is equipped with the coarsest topology that makes all maps
\[\sign_\bullet q \colon  \Sper R \rightarrow \Z,\ \alpha \mapsto \sign_\alpha 
q\]
continuous, for $q \in W(R)$.
 Furthermore, if $\alpha \subseteq \beta \in \Sper R$, we have $\sign_\alpha
= \sign_\beta$, and in particular $\sign_\alpha$ is equal to $\sign_{\alpha^*}$,
the signature at the unique maximal ordering that contains $\alpha$.
Note that $\sign_\alpha$ and $\sign_\beta$ are only equal on nonsingular forms,
and will be different for general forms if $\alpha \subsetneq \beta$ (indeed, if $a
\in \beta \setminus \alpha$, we clearly have that $\sign_\beta \qf{a} \in
\{0,1\}$, while $\sign_\alpha \qf{a} = -1$).

The map
\begin{equation}\label{sperm-cont}
  \Sperm R \rightarrow \Sign R,\ \alpha \mapsto \sign_\alpha,
\end{equation}
is continuous and surjective.

If $R$ is semilocal, a basis for the
topology on $\Sign R$ is given by the sets
\[H(u_1,\ldots,u_k) := \{\tau \in \Sign R \mid \tau(u_1) = \cdots = 
\tau(u_k) =1\},
\]
for $k \in \N$ and $u_1, \ldots, u_k \in R^\x$. Furthermore, if  $R$ is
semilocal, then the map
\eqref{sperm-cont} is a
homeomorphism, so that we have
\begin{equation}\label{for-the-end}
  \begin{aligned}
    \xymatrix@R=0.5ex{
      \Sper R\ \ar[r] & \Sperm R \ar[r]^-\sim & \Sign R \\
      \alpha \ar@{|->}[r] & \alpha^* \ar@{|->}[r] & \sign{\alpha^*} = \sign_\alpha
    }
  \end{aligned},
\end{equation}
where all maps are continuous. In particular, for $\alpha, \beta \in \Sper R$,
$\sign_\alpha = \sign_\beta$ if and only if $\alpha^* = \beta^*$, which implies
$\alpha \subseteq \beta$ if $\beta \in \Sperm R$.
\medskip

The following lemma gathers two straightforward reformulations of some of
Knebusch's results from \cite{kne75} and \cite{knebusch84}:

\begin{lemma}\label{separation}\mbox{}
  \begin{enumerate}[{\quad\rm (1)}]
    \item Let $\vf : R \rightarrow T$ be a finite \'etale extension of commutative
      rings, and let $\alpha \in
      \Sperm R$. Then the number of extensions of $\alpha$ to $\Sperm T$ is
      finite.
    \item Assume that $R$ is semilocal, let $t \in \N$,
     and let $\beta_1, \ldots, \beta_t \in
      \Sperm R$  all be different. Then there are $k \in
      \N$ and $r_1, \ldots, r_k \in R^\x$ such that $\{\beta_1, \ldots,
      \beta_t\} \cap \mathring H(r_1, \ldots, r_k) = \{\beta_1\}$.
  \end{enumerate}
\end{lemma}

\begin{proof}
(1) Let $\vf^* : W(R) \rightarrow W(T)$ be the map induced by $\vf$ on the
      Witt rings of $R$ and $T$. 
      Let $\gamma \in \Sperm T$ be such that $\alpha =
      \vf^{-1}(\gamma)$. Then $\sign_\alpha = \sign_\gamma \circ \vf^*$ by
      Lemma~\ref{same-sign} below,
      and $\sign_\gamma$ is a signature of $T$ that extends
      $\sign_\alpha$. The result follows since there are only finitely many such
      signatures by \cite[Theorem~3.4]{kne75}.

(2) It suffices to prove the lemma for $t=2$. Assume that the result does
      not hold. Then by \cite[Theorem~2.19]{A-U-Az-PLG} we obtain $\beta_1 =
      \beta_2$, contradiction.
\end{proof}

\begin{lemma}\label{same-sign}
  Let $\vf : R \rightarrow T$ be a morphism of commutative rings, let $\gamma
  \in \Sper T$ and let $\alpha := \vf^{-1}(\gamma) \in \Sper R$. Then
  $\sign_\alpha = \sign_\gamma \circ \vf^{*} : W(R)\to \Z$.
\end{lemma}
\begin{proof}
  Let $q$ be a nonsingular symmetric bilinear form over $R$. Let $\fp := \Supp \alpha$,
  and observe that $2 \not \in \fp$, so that $2$ is invertible in $R_\fp$.
  It is well known that $q \ox R_\fp$ is diagonalizable in $R_\fp$ in this case
  (cf. \cite[Chapter~II, Remark~4.6.5]{knus91}): $q \ox R_\fp \simeq \qf{a_1,
  \ldots, a_t}$ with $a_1, \ldots, a_t \in R_\fp$. Up to multiplying each $a_i$ by the square of an element in
  $R\setminus \fp$, we can assume that the $a_i$
  are in $R$, and therefore in $R \setminus \fp$ since $q \ox R_\fp$ is
  nonsingular.

  If $\ve_i \in \{-1,1\}$ is the sign of $a_i$ with respect to
  $\alpha$ we have $\sign_\alpha q = \ve_1 + \cdots + \ve_t$. In particular, to
  show that $\sign_\alpha q = \sign_\gamma \vf^*(q)$ it suffices to show, for $a
  \in R \setminus \fp$, that $a \in \alpha \setminus -\alpha$ if and only if
  $\vf(a) \in \gamma \setminus -\gamma$, but this follows immediately from the
  definition of $\alpha$.
\end{proof}

\subsection{Signatures of hermitian forms}
  
We recall the following definition from \cite{A-U-Az-PLG}:

\begin{defn}\label{def:sig_h}
  Let $h$ be a hermitian form over $(A,\s)$ and let $\alpha\in \Sper R$.  Then
  $h\ox \kappa(\alpha)$ is a hermitian form over the 
  $\kappa(\alpha)$-algebra with
  involution $(A(\alpha), \s(\alpha))$, and we define the \emph{$\CM$-signature
  of $h$ at $\alpha$} by
  \[\sign^{\CM}_\alpha h := \sign^{\CM_{\bar\alpha}}_{\bar \alpha} 
  (h \ox_R \kappa(\alpha)),\]
  where $\CM_{\bar\alpha}$ is a Morita equivalence as in
  \cite[Section~3.2]{A-U-Kneb}; see \cite[Definition~3.1]{A-U-Az-PLG} and the
  discussion following it.
\end{defn}

The main drawback of the $\CM$-signature at $\alpha \in \Sper R$ is that there
is no canonical choice of Morita equivalence $\CM_{\bar\alpha}$. Another 
choice of Morita equivalence will lead at most to a change of sign in the
signature of a form (cf.  \cite[Proposition~3.4]{A-U-Kneb}), while an arbitrary 
change in the sign of this signature can
be obtained by taking another, well-chosen, Morita equivalence (cf.  
\cite[first paragraph of Section~3.3]{A-U-Kneb}). We therefore adopt the 
following:
\medskip

\noindent\textbf{Convention:}
  If there are several instances of $\CM$-signatures of forms over $(A,\s)$ in
  a statement or formula, we use the same choice of Morita equivalence for each
  of them.
\medskip

This lack of a canonical choice for $\CM_{\bar \alpha}$ is in particular a problem
if we want to consider the total signature of a form as a continuous 
function on $\Sper R$. We solved this problem with the definition of $\eta$-signatures in 
\cite{A-U-Az-cont} by introducing a so-called reference form $\eta$, as follows:

\begin{defn}\mbox{}
  \begin{enumerate}[{\quad\rm (1)}]
    \item We define $\Nil[A,\s] := \{\alpha \in \Sper R \mid \sign^\CM_\alpha =
      0\}$. It is a clopen subset of $\Sper R$ (see
      \cite[Theorem~3.11]{A-U-Az-cont}).
    \item We say that a hermitian form $\eta$ is a \emph{reference form for $(A,\s)$}
      if $\sign^\CM_\alpha \eta \not = 0$ for every $\alpha \in \Sper R \setminus
      \Nil[A,\s]$.
  \end{enumerate}
\end{defn}
For $\alpha \in \Sper R \setminus \Nil[A,\s]$, we can now use $\eta$ to fix the
sign in the computation of $\sign^\CM_\alpha$: We choose a Morita equivalence
$\CM_{\bar \alpha}$ such that $\sign^\CM_\alpha \eta > 0$. We call this
signature the $\eta$-signature at $\alpha$, denoted $\sign^\eta_\alpha$.
Equivalently, it can be obtained as follows, for any choice of the Morita
equivalence $\CM_{\bar \alpha}$:
\[\sign^\eta_\alpha h := (\sgn \sign^\CM_\alpha \eta) \sign^\CM_\alpha h,\]
for $\alpha \in \Sper R \setminus \Nil[A,\s]$ and $h$ hermitian form over
$(A,\s)$, and where $\sgn$ denotes the sign function.

Note that since all signatures at $\alpha$ are computed by scalar extension to
$\kappa(\alpha)$, we have
\[\sign^\eta_\alpha h = \sign^{\eta \ox \kappa(\alpha)}_{\bar \alpha} (h \ox
\kappa(\alpha)).\]

\begin{rem}\label{signprop}
  The $\CM$-signature $\sign^{\CM}_\alpha$, and in particular the
  $\eta$-signature $\sign^\eta_\alpha$ (which is a special case of it, for a
  particular choice of $\CM_{\bar \alpha}$) inherit all the properties from the
  map $\sign^{\CM_{\bar\alpha}}_{\bar \alpha}$ defined for central simple
  algebras with involution. For instance, it is additive
  and $\sign^\CM_\alpha (q\ox h)=\sign_\alpha q \cdot \sign^\CM_\alpha h $ if
  $q$ is quadratic over $R$ and $h$ hermitian over $(A,\s)$, cf.
  \cite[Proposition~3.6]{A-U-Kneb}.
\end{rem}

\section{The Knebusch trace formula}\label{sec:ktf}

\subsection{Finite \'etale extensions and real closed
fields}

We recall Knebusch's trace formula for symmetric bilinear forms over commutative
rings, which is proved in general for Frobenius extensions of rings in
\cite{kne77}, but we only state it for the special case of finite \'etale
extensions, cf. \cite[Final paragraph of page 168]{kne77}.

Let $\vf : R \rightarrow T$ be a finite \'etale extension and consider the
induced map $\vf^*: W(R)\to W(T)$.  If $\rho$ is a signature of $R$, we denote
by $S(\vf, \rho)$ the set of all signatures on $T$ that extend $\rho$ with
respect to $\vf$, i.e., $S(\vf,\rho)$ is the set of signatures $\tau: W(T)\to
\Z$ such that $\tau(\vf^*(q)) = \rho(q)$ for all $q \in W(R)$. We then have:

\begin{thm}[{\cite[Theorem~3.4]{kne75}, \cite[Theorem~1.1]{kne77}}] \label{KTF-qf}
  Let $\rho$ be a signature of $R$. Then $S(\vf, \rho)$ is finite and there is a
  map $n: S(\vf, \rho) \rightarrow
  \N \cup \{0\}$ such that $n(\tau) > 0$ for all $\tau \in S(\vf, \rho)$, and
  \begin{equation}\label{eq-ktf}
    \rho(\Tr_{T/R}^*(q)) = \sum_{\tau \in S(\vf,\rho)} n(\tau) \tau(q),
  \end{equation}
  for all $q \in W(T)$.

  Moreover, if $R$ and thus $T$ (cf.
  \cite[Chapter~VI, Proposition~1.1.1]{knus91}) are semilocal, then $n(\tau)
  = 1$ for all $\tau \in S(\vf, \rho)$.
\end{thm}

\begin{proof}
  The first part can be found in \cite[Theorem~3.4]{kne75} (or 
  \cite[Theorem~1.1]{kne77} for the case of Frobenius extensions).
  The fact that $n(\tau)=1$ in the semilocal case is first stated in \cite[p.~74
  (second paragraph before Lemma~3.11)]{kne75} and then proved in 
  \cite[Section~8]{kne76}.
\end{proof}

We recall the construction used in the proof of Theorem~\ref{KTF-qf}. We will use it
several times in this paper.

Let $\alpha\in \Sper R$. By Lemma~\ref{trace-extension} the following diagram
commutes:
\begin{equation}\label{diag1}
  \begin{split}
  \xymatrix{
    T \ar[d]^{\Tr_{T/R}} \ar[rr]^-{\ox_R k(\alpha)} & & T \ox_R k(\alpha)
      \ar[d]^{\Tr_{T \ox_R k(\alpha)/k(\alpha)}} \\
    R \ar[rr]^{\ox_R k(\alpha)} & & k(\alpha)
  }
  \end{split}.
\end{equation}
Note that $T \ox_R k(\alpha)$ is well-defined since $T$ is a left $R$-module 
via $\vf$ and $k(\alpha)$ is a right $R$-module via $r_\alpha$. In particular,
for $x \in R$, 
\begin{equation}\label{faber}
  \vf(x)\ox 1 = 1\ox r_\alpha(x)\text{ in } T \ox_R k(\alpha).
\end{equation}
Since $T \ox_R k(\alpha)$ is a separable $k(\alpha)$-algebra (cf.
\cite[Corollary~4.3.2]{ford17}), it follows from \cite[Corollary~4.5.8]{ford17}
that there is a $k(\alpha)$-algebra isomorphism  
  \[
  \omega: T \ox_R k(\alpha) \to E_1
  \times \cdots \times E_t,
  \] 
where $E_1,\ldots, E_r$ are equal to $k(\alpha)$ (for some $0\leq r\leq t$) and
$E_{r+1}, \ldots, E_t$ are equal to $k(\alpha)(\sqrt{-1})$. We thus obtain the
following commutative diagram (the second square is commutative by \cite[(11) on
p.~137]{bourbaki-8}):
\begin{equation}\label{poodle}
  \begin{split}
    \xymatrix{
      T \ox_R k(\alpha) \ar[r]^--{\omega} \ar[d]^--{\Tr_{T\ox_R k(\alpha)
      / k(\alpha)}}&  E_1\x\cdots \x E_t \ar[r]^--{\prod p_i}
      \ar[d]^--{\Tr_{E_1\x\cdots\x E_t / k(\alpha)}} & E_1 \x \cdots \x
      E_t \ar[d]^--{\sum_{i=1}^t \Tr_{E_i / k(\alpha)}} \\
      k(\alpha) \ar[r]^--{\id} & k(\alpha) \ar[r]^--{\id}  & k(\alpha)
    }
  \end{split}.
\end{equation}
For $j=1, \ldots, t$ we define the maps
\begin{equation}\label{beta}
  \beta_j : T \rightarrow  T \ox_R
    k(\alpha)  \xrightarrow{\omega}  \prod_i E_i 
    \xrightarrow{p_j} E_j,
\end{equation}
where $p_j$ denotes the canonical projection. Note that for $1\leq i \leq r$
the diagram
\begin{equation}\label{comm-beta}
  \begin{split}
    \xymatrix{
      T \ar[r]^{\beta_i} & k(\alpha) \\
      R \ar[u]_{\vf} \ar[ur]_{r_\alpha} &
    }
  \end{split}
\end{equation}
commutes by Equation~\eqref{faber}.

\begin{prop}\label{extend-nil}
  Let $\alpha \in \Sper R$.
  \begin{enumerate}[{\quad \rm (1)}]
    \item We have $\sign_\alpha \Tr_{T/R}^* \qf{1} = r$ (the number of $E_i$ that are real closed).
    \item Assume that $T$ has odd rank $r_0$ at $\Supp \alpha$ as an $R$-algebra
      (via the map $\vf$). Then there is $\gamma \in \Sper T$ such that $\alpha
      = \vf^{-1}(\gamma)$, i.e., $\gamma$ is an extension of $\alpha$
      (over $\vf$).
  \end{enumerate}
  Assume now that $\gamma \in \Sper T$ is such that $\alpha = \vf^{-1}(\gamma)$.
  Then:
  \begin{enumerate}[{\quad\rm (1)}]\setcounter{enumi}{2}
    \item We have $\alpha \in \Nil[A,\s]$ if and only if $\gamma \in 
      \Nil[A\ox_\vf T,\s\ox\id]$.

    \item If $\eta$ is a reference form for $(A,\s)$, then $\eta \ox_\vf T$ is a
      reference form for $(A \ox_\vf T, \s \ox \id)$, and
      $\sign^\eta_\alpha h = \sign^{\eta \ox T}_\gamma h \ox T$ for every
      hermitian form $h$ over $(A,\s)$.
  \end{enumerate}
\end{prop}
\begin{proof}
  (1) We apply the same argument as in the proof of the Knebusch trace formula 
  to
  the form $\qf{1}$ over $T$ in order to compute $\sign_\alpha \Tr_{T/R}^*
  \qf{1}$. By definition, $\sign_\alpha \Tr_{T/R}^* \qf{1} = \sign_{\tilde
  \alpha} \bigl( (\Tr_{T/R}^* \qf{1}) \ox_R k(\alpha)\bigr)$, where $\tilde\alpha$
  denotes the unique ordering on $k(\alpha)$,
  and we obtain
  \[\sign_{\tilde \alpha} \bigl( (\Tr_{T/R}^* \qf{1}) \ox_R k(\alpha)\bigr) = \sign_{\tilde
  \alpha} \Tr_{T \ox_R k(\alpha) / k(\alpha)}^* \bigl(\qf{1} \ox_R
  k(\alpha)\bigr)\]
  by pushing the form $\qf{1}$ through Diagram~\eqref{diag1}.
  Therefore, to compute this signature, we consider the form $\qf{1} \ox
  k(\alpha)$ over $T \ox_R k(\alpha)$, push it through  
  Diagram~\eqref{poodle}, and
  compute the signature.
  It follows that
  \[\sign_\alpha \Tr_{T/R}^* \qf{1} = \sum_{i=1}^t \sign_{\tilde\alpha}
  \Tr_{E_i/k(\alpha)}^* \qf{1}.\]
  Using now that $\Tr_{E_i/k(\alpha)} = \id$ if $1 \le i \le r$, and that $\sign
  (\Tr_{E_i/k(\alpha)}^* b) = 0$ for any nonsingular symmetric bilinear form $b$ over $E_i$
  if $r+1 \le
  i \le s$ (since $E_i$ is algebraically closed, and so $b$ must be weakly hyperbolic, which implies that $\Tr_{E_i/k(\alpha)}^* b$ is weakly hyperbolic by Section~\ref{inv-trace}), we obtain $\sign_\alpha
  \Tr_{T/R}^* \qf{1} = r$.

  (2) Since $\sign_\alpha \Tr_{T/R}^* \qf{1} \not = 0$ (it is equal to $r_0$
  mod 2), we obtain that $r\geq 1$ and that the map $\beta_1 : T \rightarrow
  k(\alpha)$ exists. 
  Define $\gamma := \beta_1^{-1}(\tilde\alpha) \in \Sper T$. Since $\alpha =
      r_\alpha^{-1}(\tilde \alpha)$, 
      the commutativity of Diagram~\eqref{faber}  yields $\alpha
      = \vf^{-1}(\gamma)$.
      
  (3) Since $\alpha = \vf^{-1}(\gamma)$, $\vf$ induces a map
    $R/\Supp \alpha \rightarrow T/\Supp \gamma$ 
    and, taking fractions, a map
    $\kappa(\alpha) \rightarrow \kappa(\gamma) $ such that the diagram
    \begin{equation*}
      \xymatrix{
      R \ar[r]^\vf \ar[d] & T\ar[d] \\
      \kappa(\alpha)\ar[r] & \kappa(\gamma)
      }
    \end{equation*}
    commutes. By \cite[Theorem~1.11.2]{KSU} there is a (unique) morphism
    $\iota : k(\alpha) \rightarrow k(\gamma)$ such
    that the following diagram also commutes:
    \begin{equation}\label{vf-commutes}
      \begin{split}
      \xymatrix{
      R \ar[r]^\vf \ar[d]_{r_\alpha}& T \ar[r]^{r_\gamma} & k(\gamma) \\
      k(\alpha) \ar[urr]_\iota & & 
      }
      \end{split}.
    \end{equation}
    It follows that
    \begin{equation}\label{eq-tensors}
    \begin{aligned}
      ((A \ox_\vf T) \ox_{r_\gamma} k(\gamma), (\s \ox_\vf \id_T) 
      &\ox_{r_\gamma} \id_{k(\gamma)}) \\
      &\cong (A
        \ox_{r_\gamma \circ \vf} k(\gamma), \s \ox \id) \\
        &\cong (A \ox_{\iota \circ r_\alpha} k(\gamma), \s \ox \id) \\
        &\cong ((A \ox_{r_\alpha} k(\alpha)) \ox_\iota k(\gamma), (\s \ox_{r_\alpha} \id) \ox_{\iota} \id).
    \end{aligned}
    \end{equation}
    Observe that:
    \begin{enumerate}[\quad (a)]
      \item The involution $(\s \ox_\vf \id_T) \ox_{r_\gamma} \id_{k(\alpha)}$
      is of the same type as $(\s \ox_{r_\alpha} \id_{k(\alpha)})
       \ox_{\iota} \id_{k(\gamma)}$;
     \item Using the notation of \cite[Lemma~3.10]{A-U-Az-cont}, Equations~\eqref{eq-tensors}
       give that
        $\alpha\in C_i$ implies $\gamma \in C_i $ for $i=4,\ldots,8$, and that 
        this
      implication is an equivalence since $C_4,\ldots, C_8$ cover all
      possibilities.
    \end{enumerate}
     Therefore, $\alpha \in \Nil[A,\s]$ if and only if 
     $\gamma \in \Nil[A \ox T, \s \ox \id]$ by \cite[Remark~3.6, or
     its formulation as Equation~(3.3) in the proof of
     Theorem~3.11]{A-U-Az-cont}.

   (4) For the first statement we need to show that $\sign^\CM_\gamma \eta \ox T
   \not = 0$ for all $\gamma \in \Sper T \setminus \Nil[A \ox_\vf T, \s \ox
   \id]$. Let $\gamma \in \Sper T \setminus \Nil[A \ox_\vf T, \s \ox \id]$. By
   part~(3),  $\alpha := \vf^{-1}(\gamma) \not \in \Nil[A,\s]$, so that
   $\sign^\CM_\alpha \eta \not = 0$. From Diagram~\eqref{vf-commutes} we obtain,
   for any hermitian form $h$ over $(A,\s)$,
     \begin{equation}\label{sign-tensor}
       \begin{aligned}
         \sign^\CM_\gamma h \ox T &= \sign^\CM_{\tilde\gamma} 
         h \ox T \ox k(\gamma) 
            \quad \text{(by definition of signature)} \\
           &= \sign^\CM_{\tilde\gamma} h \ox k(\alpha) \ox k(\gamma) \\
           &= \ve \sign^\CM_{\tilde\alpha} h \ox k(\alpha) 
             \text{ for some $\ve \in \{-1,1\}$ 
             by \cite[Lemma~3.8]{A-U-Kneb}} \\
           &= \ve \sign^\CM_\alpha h.
       \end{aligned}
     \end{equation}
     Applying this with $h = \eta$ we get $\sign^\CM_\gamma \eta \ox T = \ve
     \sign^\CM_\alpha \eta \not = 0$.

     For the second statement, we use Equations~\eqref{sign-tensor} twice, for $\eta$ and
     for $h$, and obtain (for the second line; the first  and final lines
     are the definitions)
     \begin{align*}
       \sign^{\eta \ox T}_\gamma (h \ox T) &= \sgn(\sign^\CM_\gamma \eta \ox T)
         \sign^\CM_\gamma (h \ox T) \\
         &= \sgn(\sign^\CM_\alpha \eta) \sign^\CM_\alpha h \\
         &= \sign^\eta_\alpha h. \qedhere
     \end{align*}
\end{proof}

\begin{lemma}\label{fse}\mbox{}
  Assume that $R$ is semilocal. Then

  \begin{enumerate}[{\quad \rm (1)}]
    \item $|S(\vf,\rho)| = \rho (\Tr^*_{T/R} \qf{1})$ for all signatures $\rho$
      of $R$;
      
    \item Let $\alpha\in \Sper R$. With reference to Equation~\eqref{beta},
    $\sign \circ \beta_1^*, \ldots, \sign
      \circ \beta_r^*$ are all different and they are all the signatures of $T$ that
      extend $\sign_\alpha$.
  \end{enumerate}
\end{lemma}

\begin{proof}
(1) Apply Equation~\eqref{eq-ktf} to the symmetric bilinear 
  form $\qf{1}$, using the fact that $n(\tau)=1$ for all $\tau \in S(\vf,\rho)$
  because $R$ and $T$ are semilocal ($T$ is semilocal by \cite[Chapter~VI,
  Proposition~1.1.1]{knus91}):
  \[
    \rho(\Tr_{T/R}^* \qf{1}) = \sum_{\tau \in S(\vf,\rho)} \tau(\qf{1})
        = |S(\vf, \rho)|.
  \]
(2) We start with:   
\smallskip

\emph{Fact.} $\beta_1, \ldots, \beta_r$ are all the morphisms from $T$ to $k(\alpha)$
such that $\beta_i \circ \vf = r_\alpha$.
\smallskip

\emph{Proof of the Fact.} The property $\beta_i \circ \vf = r_\alpha$ is Equation~\eqref{comm-beta}.  
The argument for the remaining claim appears in \cite[Proof of Proposition~1.6, as
observed on p.~74, line 14--15]{kne75}, but we provide one here since Knebusch's
results are much more general and involve slightly different notation. Let $\lambda : T
\rightarrow k(\alpha)$ be such that $\lambda \circ \vf = r_\alpha$. A direct
verification shows that $\lambda$ yields a morphism $\lambda' : T \ox_R 
k(\alpha)\rightarrow k(\alpha)$, induced by $t \ox z \mapsto \lambda(t)z$ (recall that $R$ acts on $T$ on the right via $\vf$ and on $k(\alpha)$ on the left via $r_\alpha$). 
We therefore have a
$k(\alpha)$-morphism
$\lambda' \circ \omega^{-1} : E_1 \x \cdots \x E_t \rightarrow k(\alpha)$, which
must be (considering its kernel) equal to $p_i$ for some $i \in \{1, \ldots,
r\}$. Thus, for $a \in T$,
\[
  \lambda(a) = \lambda'(a \ox 1) = \lambda' \circ \omega^{-1}(\omega(a \ox 1))
  = p_i(\omega(a \ox 1)) = \beta_i(a).
\]
End of the proof of the Fact.

  Since $\beta_j \circ \vf=r_\alpha$, 
  each $\beta_j$ defines a
  signature $W(T)\rightarrow \Z$ that extends $\sign_\alpha$, i.e., 
  \[
  \sign \circ \beta_j^* \circ \vf^* = \sign \circ r_\alpha^*.
  \] 
  Following \cite[p.~74]{kne75} we define, for a fixed signature $\tau: 
  W(T)\to \Z$, 
  \[
    n(\tau):= |\{i\in \{1,\cdots, r\} \mid \sign \circ \beta_i^*=\tau\}|
  \]
  (it is the $n(\tau)$ from Theorem~\ref{KTF-qf}).  As recalled in
  Theorem~\ref{KTF-qf},  $n(\tau)=1$ since $R$ and therefore $T$ are semilocal.
  This means that the maps $\sign \circ \beta_1^*, \ldots, \sign \circ
  \beta_r^*$ are different signatures on $T$. By (1) the number of
  signatures on $T$ extending $\sign_\alpha$ is equal to $|S(\vf,
  \sign_\alpha)|=\sign_\alpha \Tr_{T/R}^* \qf{1}$, which is equal to $r$ by
  Proposition~\ref{extend-nil}(1).  Therefore, $\sign_\alpha \circ \beta_i^*$
  for $1\leq i\leq r$ are all the signatures on $T$ that extend $\sign_\alpha$.
\end{proof}

\subsection{The trace formula}

The following result is a hermitian version of the Kne\-busch trace formula, 
Theorem~\ref{KTF-qf}.

\begin{thm}[Knebusch trace formula]\label{ktf}
  Let $\alpha \in \Sper R$, and let $\vf: R \rightarrow T$ be a 
  finite \'etale extension of $R$. Then there are $\gamma_1, \ldots, \gamma_r
  \in \Sper T/\alpha$ (namely the preimages of the unique ordering $\tilde\alpha$ on
  $k(\alpha)$ under the maps $\beta_1, \ldots, \beta_r$ introduced in
  Equation~\eqref{beta}) such that, for every hermitian form $h$ over $(A\ox_R T,
  \s\ox\id)$, we have
  \begin{equation}\label{Az-KTF}
    \sign^\eta_{\alpha} (\Tr^{*}_{A\ox_R T/A} h) =\sum_{i=1}^r \sign^{\eta\ox
    T}_{\gamma_i} h,
  \end{equation}
  where $\Tr_{A\ox_R T/A}:=\id_A \ox_R \Tr_{T/R}:A\ox_R T \to A$ is the map
  induced by the trace map $\Tr_{T/R}$, and
  $\Tr^*_{A\ox_R T/A}$ is the induced map on hermitian forms.

  Moreover, if $R$ is semilocal, then $\gamma_1^*,\ldots, \gamma_r^*$ are all different, and $\Sperm T/\alpha = \{\gamma_1^*,\ldots, \gamma_r^*\}$.
\end{thm}

\begin{proof} 
  The commutative Diagram~\eqref{diag1} induces the commutative diagram 
  \begin{equation}
    \begin{aligned}
    \xymatrix{
      A\ox_R T \ar[d]^{\Tr_{A\ox_R T / A}} \ar[rr]^-{\ox_R k(\alpha)} & & A \ox_R T \ox_R k(\alpha)
        \ar[d]^{\id_A\ox_R \Tr_{T\ox_R k(\alpha)/k(\alpha)}} \\
      A \ar[rr]^{\ox_R k(\alpha)} & & A \ox_R k(\alpha)
    }
    \end{aligned}.
  \end{equation}
  We have, first by definition of signature, and then by commutativity of
  the above diagram (when pushing the form $h$ through it):
  \begin{equation}\label{sign_1}
    \begin{aligned}
      \sign^\eta_\alpha (\Tr^*_{A\ox_R T/A} h) &= \sign^{\eta \ox
      k(\alpha)}_{\tilde\alpha} \bigl( (\Tr^*_{A\ox_R T/A} h)\ox_R k(\alpha)\bigr)\\
      &=\sign^{\eta \ox k(\alpha)}_{\tilde\alpha} (\id_A\ox_R \Tr_{T\ox_R k(\alpha) / k(\alpha)})^*
      (h\ox_R k(\alpha)).
    \end{aligned}
  \end{equation}
  With reference to the notation introduced after the proof of 
  Theorem~\ref{KTF-qf}
  and Diagram~\eqref{poodle}, we consider the commutative diagram 
  \begin{equation*}
    \begin{split}
  \xymatrix{
    A\ox_R T \ox_R k(\alpha) \ar[r]^--{\id_A \ox \omega} \ar[d]^--{\id_A  \ox_R
    \Tr_{T\ox_R k(\alpha) / k(\alpha)}}&  A  \ox_R \prod_{i=1}^t E_i  
    \ar[rr]^--{\prod (\id_A\ox p_i) } \ar[d]^--{\id_A \ox  \Tr_{E_1\x\cdots\x E_t / k(\alpha)}
    } & & (A\ox_R E_1)\x \cdots \x (A\ox_R E_t)
    \ar[d]^--{\sum_{i=1}^t \id_A \ox \Tr_{E_i / k(\alpha)} } \\
    A\ox_R  k(\alpha) \ar[r]^--{\id} &A\ox_R k(\alpha) \ar[rr]^--{\id}  &  & A\ox_R
    k(\alpha)\\
  }\end{split}.
  \end{equation*}
 We push the form $h\ox k(\alpha)$ 
  through this diagram:
  \begin{equation*}
    \begin{split}
  \xymatrix{
    h\ox k(\alpha) \ar@{|->}[r]^--{(\id_A\ox\omega)^*} \ar@{|->}[d] & h' \ar@{|->}[r] & \Bigl((\id_A\ox p_1)^* (h'), \ldots, (\id_A\ox p_t)^*(h')\Bigr) 
    \ar@{|->}[d]\\
    (\id_A\ox \Tr_{T\ox k(\alpha) / k(\alpha)})^* (h\ox k(\alpha))
    \ar@{|->}[rr]^--{\id} & & \sum_{i=1}^t (\id_A\ox \Tr_{E_i / k(\alpha)})^*(h'_i)
  }
  \end{split},
  \end{equation*}
  where $h'_i:= (\id_A\ox p_i)^* (h')=  (\id_A \ox \beta_i)^* (h)$.
  Thus, after applying signatures to the bottom of this diagram, we obtain (the
  first equality is from Equation~\eqref{sign_1}):
  \begin{equation}\label{sign_3}
    \begin{aligned}
      \sign^\eta_{\alpha} (\Tr^*_{A\ox_R T/A} h) &= 
      \sign^{\eta \ox k(\alpha)}_{\tilde\alpha} 
          (\id_A\ox \Tr_{T\ox k(\alpha) / k(\alpha)})^*(h\ox_R k(\alpha)) \\
        &= \sum_{i=1}^t \sign^{\eta \ox k(\alpha)}_{\tilde\alpha} 
          (\id_A\ox \Tr_{E_i / k(\alpha)})^*(h'_i).
    \end{aligned}
  \end{equation}
  We need to consider separately the cases where $E_i$ is algebraically  closed 
  and where
  $E_i$ is real closed:
  \medskip
  
  \emph{Case 1.} Assume that $r+1\leq i \leq t$. In this case $E_i$ is algebraically closed. By
Lemma~\ref{invo-trace} below, $h'_i \simeq \psi_w \perp \psi_0$, with $\psi_w$ weakly
hyperbolic and $\psi_0$ a zero form, so that
\[(\id_A\ox \Tr_{E_i / k(\alpha)})^*(h'_i) \simeq \underbrace{(\id_A\ox \Tr_{E_i /
k(\alpha)})^*(\psi_w)}_{\psi_w'} \perp \underbrace{(\id_A\ox \Tr_{E_i /
k(\alpha)})^*(\psi_0)}_{\psi_0'}.\]
The form $\psi_w'$ is weakly hyperbolic, cf. Section~\ref{inv-trace}, and
$\psi_0'$ is clearly a zero form. Both forms have zero signature at any
ordering, and thus
\[
  \sign^{\eta \ox k(\alpha)}_{\tilde\alpha} (\id_A\ox \Tr_{E_i / 
  k(\alpha)})^*(h'_i)=0.
\]
    
  \emph{Case 2.} Assume that $1\leq i \leq r$. Then
  $E_i=k(\alpha)$, and hence $\Tr_{E_i / k(\alpha)}=\id_{k(\alpha)}$. Define
  $\gamma_i := \beta_i^{-1}(\tilde\alpha) \in \Sper T$. Then, by
  \cite[Proposition~5.9]{A-U-Az-cont}, for every hermitian form $\xi$ over
  $(A \ox T, \s \ox \id)$, we have
  \begin{equation}\label{horrible}
    \sign^{\eta \ox T}_{\gamma_i} \xi = \sign^{\eta \ox T 
    \ox_{\beta_i} k(\alpha)}_{\tilde\alpha} (\xi \ox_{\beta_i} k(\alpha)).
  \end{equation}
  Therefore,
    \begin{align*}
      \sign^{\eta \ox k(\alpha)}_{\tilde\alpha} 
      (\id_A\ox \Tr_{E_i / k(\alpha)})^*(h'_i) 
      &= \sign^{\eta \ox k(\alpha)}_{\tilde\alpha} h'_i\\
      &= \sign^{\eta \ox k(\alpha)}_{\tilde\alpha} (\id_A \ox \beta_i)^* (h)\\
      &= \sign^{\eta \ox T \ox_{\beta_i} k(\alpha)}_{\tilde\alpha} 
      (h \ox_{\beta_i} k(\alpha))\\
      &= \sign^{\eta\ox T}_{\gamma_i} h \quad \text{(by Equation~\eqref{horrible})}.
    \end{align*}
  This proves the first part of the theorem.
  
  We now consider the case where $R$ is 
  semilocal.
  Observe that $\sign_{\gamma_i} = \sign_{\tilde\alpha} \circ \beta_i^*$ on 
  $W(T)$ for $i=1,
  \ldots, r$ by definition of $\gamma_i$. Thus, by Lemma~\ref{fse}(2), $\sign_{\gamma_1}, \ldots,
  \sign_{\gamma_r}$ are all the signatures of $T$ that extend $\alpha$. Since
  Diagram~\eqref{comm-beta} commutes, the definition of $\gamma_i$ gives $\alpha
  \subseteq \vf^{-1}(\gamma_i)$, and by maximality of $\alpha$ we have $\alpha =
  \vf^{-1}(\gamma_i)$, i.e., $\gamma_i \in \Sper T/\alpha$. Now let $1 \le i
  \not = j \le r$. Since $\sign_{\gamma_i} \not = \sign_{\gamma_j}$, we obtain
  that $\gamma_i^* \not = \gamma_j^*$ (indeed, with reference to Diagram~\ref{for-the-end}, 
  if $\gamma_i^* =\gamma_j^*$,
  then $\sign_{\gamma_i}=\sign_{\gamma^*_i}=\sign_{\gamma^*_j}=
  \sign_{\gamma_j}$, contradiction).  
  Finally, let $\gamma \in \Sper T/\alpha$.
  Then $\sign_\gamma = \sign_{\gamma_i}$ for some $i$, i.e., $\gamma^* =
  \gamma_i^*$. Therefore $\gamma_1, \ldots, \gamma_r$ are as described in the
  second part of the theorem.
\end{proof}

\begin{lemma}\label{invo-trace}
  Let $(B,\tau)$ be an Azumaya algebra with involution over an algebraically closed field $E$. Then every hermitian form over
  $(B,\tau)$ is the orthogonal sum of a weakly hyperbolic form and a zero form.
\end{lemma}
\begin{proof}
  Let $h$ be a hermitian form over $(B,\tau)$.
  By \cite[Proposition~A.3]{A-U-PS}, $h \simeq h^\ns \perp h_0$, where $h^\ns$
  is nonsingular and $h_0$ is a zero form. We consider two cases, depending
  on the kind of $\tau$:

  \emph{Case~1.} $\tau$ is of the first kind. Then $Z(B) = E$ and $B$ is a central simple
  $E$-algebra. In particular $B$ is split since $E$ is algebraically closed.
  The category of hermitian forms over $(B, \tau)$ is Morita
  equivalent to either the category of symmetric bilinear forms over $E$ (when
  $\tau$ is orthogonal) or the category of skew-symmetric bilinear forms
  over $E$ (when $\tau$ is symplectic). In each case any nonsingular form
  will be weakly hyperbolic (see for example, \cite[Section~3.6]{BPS13}), so
  $h^{\ns}$ is weakly hyperbolic.

  \emph{Case~2.} $\tau$ is of the second kind. Then $Z(B)$ is a quadratic \'etale extension
  of $E$, so $Z(B) \cong E \x E$. By \cite[Proposition~2.14]{BOI} there is a central simple $E$-algebra $C$ such that
   $(B,\tau) \cong (C \x C^\op,
  \widehat{\phantom{x}})$, where $\widehat{\phantom{x}}$ denotes the
  exchange involution. Then $h^\ns$ is hyperbolic by
  \cite[Lemma~2.1(iv)]{A-U-Kneb}.
\end{proof}

\section{A reference form of $2$-power signature} \label{sec4}

Recall that if $B$ is an Azumaya algebra, then its degree is defined
as the function $\deg B:= \sqrt{\rk_{Z(B)}B}$ on $\Spec Z(B)$. 

\begin{lemma}\label{deg2pow}
  Let $R$ be semilocal connected,  
  $(B,\tau)$ an Azumaya algebra with involution over $R$ such that $\deg
  B$ is a power of two, and  $\alpha \in \Sperm R \setminus \Nil[B,\tau]$.
  Then there is a nonsingular hermitian form $h_\alpha$ over $(B,\tau)$ such
  that $\sign^\eta_\alpha h_\alpha$ is a $2$-power.
\end{lemma}

\begin{proof}
  By \cite[Proposition~3.18]{A-U-Az-PLG} there exists an invertible element
  $u\in \Sym(B,\tau)$ such that $\sign^\CM_\alpha \qf{u}_\tau=m_{\bar\alpha} 
  (B\ox_R \kappa(\alpha), \tau\ox\id)$.
  By \cite[Corollary~3.19]{A-U-Az-PLG}  we have 
  $m_{\bar\alpha} 
  (B\ox_R \kappa(\alpha), \tau\ox\id)=n_{\alpha}$, where $n_\alpha$ is defined in
  Remark~\ref{rk-ms} (and is denoted $n_{\bar\alpha}$ in \cite[before 
  Definition~3.12]{A-U-Az-PLG}). 
  \medskip
  
  \emph{Fact.} $n_\alpha$ is a power of $2$.
  \smallskip
  
\emph{Proof of the Fact.} 
By the final line of Remark~\ref{rk-ms}, it suffices to show that 
$\dim_{k(\alpha)} B \ox_R k(\alpha)$ is a $2$-power.
Note that since $\alpha \not \in \Nil[B,\tau]$, $B 
\ox_R
  k(\alpha)$ is a simple algebra, and its centre is a field, cf. 
  \cite[Remark~2.5(4) and Remark~3.6]{A-U-Az-cont}.
  It
  follows that
  \begin{align*}
    \dim_{k(\alpha)} B \ox_R k(\alpha) &= (\dim_{Z(B \ox_R k(\alpha))} B \ox_{R}
    k(\alpha)) \cdot \dim_{k(\alpha)} Z(B \ox_R k(\alpha)) \\
    &= \rk_{Z(B)} B \cdot \dim_{k(\alpha)} Z(B \ox_R k(\alpha)).
  \end{align*}
  
  The first term in this product is a power of $2$ by hypothesis, and the second
  term is 1 or 2  (cf. Lemma~\ref{twodefs}), 
  so that $\dim_{k(\alpha)} B \ox_R k(\alpha)$ is a
  power of $2$.  End of the proof of
  the Fact.
  \medskip

  Therefore, $\sign^\CM_\alpha \qf{u}_\tau$ is a 2-power and, up to replacing
  $u$ by $-u$, we obtain that 
  $\sign^\eta_\alpha \qf{u}_\tau$ is a $2$-power for the fixed reference form
  $\eta$. Taking $h_\alpha:=\qf{u}_\tau$ and observing that it is nonsingular
  since $u$ is invertible finishes the proof.
\end{proof}

\begin{thm}\label{great-stuff}
  Let $R$ be semilocal connected and let $(B,\tau)$ be an Azumaya algebra with
  involution over $R$ such that $A$ and $B$ are Brauer equivalent and
  $\s|_S=\tau|_S$, where $S:=Z(A)=Z(B)$. Assume that $S$ is connected and that
  $\s$ and $\tau$ are of the same type. Let $\CF_\vt$ be the functor from Theorem~\ref{morprop1}.

  Then $\Nil[A,\s]=\Nil[B,\tau]$. Furthermore,
  $\CF_\vt(\eta)$ is a reference form for $(B,\tau)$ and, for any hermitian form $h$ over
  $(A,\s)$ and any $\alpha \in \Sper R$, we have
  \[\sign^\eta_\alpha h = \sign^{\CF_\vt(\eta)}_\alpha \CF_\vt(h).\]
\end{thm}

\begin{proof}
  Let $\alpha \in \Sper R \setminus \Nil[A,\s]$. We consider $(A(\alpha),
  \s(\alpha))$ and $(B(\alpha), \tau(\alpha))$. Observe that by
  Definition~\ref{def-type},  $\s$ and $\s(\alpha)$ are of the same type, and
  the same is true for $\tau$ and $\tau(\alpha)$. Therefore, $\s(\alpha)$ and
  $\tau(\alpha)$  are also of the same type. Clearly, $\kappa(\alpha)$ is
  connected, and since $\alpha \not \in \Nil[A,\s]$, we obtain by
  \cite[Remark~3.6]{A-U-Az-cont} that $S \ox_R \kappa(\alpha)$ is connected.
  Theorem~\ref{Brcom} then applies and $A(\alpha)$ is Brauer equivalent to
  $B(\alpha)$. It follows that $\Nil[A(\alpha),\s(\alpha)]=
  \Nil[B(\alpha),\tau(\alpha)]$ since these sets only depend on the Brauer class
  of the algebras and the types of the involutions, cf.
  \cite[Definition~3.7]{A-U-Kneb}. Using \cite[Lemma~3.7]{A-U-Az-PLG}, we then
  have
  \begin{align*}
    \alpha \not \in \Nil[A,\s] &\Rightarrow \bar\alpha \not \in 
    \Nil[A(\alpha),\s(\alpha)]\\
    &\Rightarrow \bar\alpha \not \in \Nil[B(\alpha),\tau(\alpha)]\\
    &\Rightarrow \alpha \not \in \Nil[B,\tau],
  \end{align*}
  so that $\Nil[B,\tau] \subseteq \Nil[A,\s]$. The other inclusion is proved
  in the same way.

  Since $\s$ and $\tau$ are of the same type, the following diagram commutes by
  Theorem~\ref{Brcom}:
  \[\begin{split}
  \xymatrix{
    \Herm(A(\alpha),
    \s(\alpha))\ar[r]^--{\CF_{\vt_{\kappa(\alpha)}}} & 
      \Herm(B(\alpha), 
  \tau(\alpha))\\
  \Herm(A,\s) \ar[r]_--{\CF_{\vt}} \ar[u]^{\CU_{(A,\s,\kappa(\alpha))}} & 
  \Herm(B,\tau) \ar[u]_{\CU_{(B,\tau,\kappa(\alpha))}}
  }
  \end{split},\]
  where the vertical arrows are scalar extensions. We thus have
  \begin{equation}\label{coffee}
  \begin{aligned}
    \sign^\eta_\alpha h &= \sign^{\eta \ox \kappa(\alpha)}_{\bar \alpha} h \ox
    \kappa(\alpha) \qquad\qquad\text{(by definition)} \\
    &= \sign^{\CU_{(A,\s,\kappa(\alpha))}(\eta)}_{\bar \alpha}
    \CU_{(A,\s,\kappa(\alpha))}(h) \\
    &= \sign^{\CF_{\vt_{\kappa(\alpha)}} \circ \CU_{(A,\s,\kappa(\alpha))}(\eta)}_{\bar \alpha}
    \CF_{\vt_{\kappa(\alpha)}} \circ \CU_{(A,\s,\kappa(\alpha))}(h)\\ 
    &\qquad\qquad\qquad\qquad\qquad\qquad\qquad\qquad\qquad\qquad \text{(by \cite[Theorem~4.2]{A-U-prime})} \\
    &= \sign^{\CU_{(B,\tau,\kappa(\alpha))} \circ \CF_{\vt}(\eta)}_{\bar \alpha}
    \CU_{(B,\tau,\kappa(\alpha))} \circ \CF_{\vt}(h) 
    \quad \text{(by the above diagram)} \\
    &= \sign^{\CF_\vt(\eta) \ox \kappa(\alpha)}_{\bar \alpha} \CF_{\vt}(h) \ox
    \kappa(\alpha).
  \end{aligned}
  \end{equation}
  We show that $\CF_\vt(\eta)$ is a reference form for $(B,\tau)$ by checking
  that $\sign^\CM_\alpha \CF_\vt(\eta) \not=0$ for every 
  $\alpha \in \Sper R\sm \Nil[B,\tau] = \Sper R\sm \Nil[A,\s]$:
  \begin{align*}
    \sign^\CM_\alpha \CF_\vt(\eta) &= \sign^{\CM_{\bar \alpha}}_{\bar \alpha}
    \CF_\vt(\eta) \ox \kappa(\alpha) \\
      &= \pm \sign^{\CF_\vt(\eta) \ox \kappa(\alpha)}_{\bar \alpha} \CF_{\vt}(\eta)  \ox
    \kappa(\alpha) \\
    &= \pm\sign^\eta_\alpha \eta  \qquad\text{(by Equations~\eqref{coffee} for $h = \eta$)} \\ 
    &\not = 0.
  \end{align*}
  Going back to the case of a general $h$, and since $\CF_\vt(\eta)$ is a
  reference form, the final line in Equations~\eqref{coffee} can be rewritten as 
  $\sign^{\CF_\vt(\eta)}_\alpha \CF_{\vt}(h)$, proving the second statement.
\end{proof}

\begin{prop}\label{sl=great}
  Let $R$ be semilocal connected and 
  let $\alpha \in \Sperm R \setminus \Nil[A,\s]$. Then there exists a nonsingular
  hermitian form $h_\alpha$ over $(A,\s)$  such that 
  $\sign^\eta_\alpha h_\alpha$ is a $2$-power.
\end{prop}

\begin{proof}
  By \cite[Theorem 8.7 and Remark 8.8]{BFP}
  there is a connected finite \'etale $R$-algebra $R_1$ of
  odd rank and an 
  Azumaya algebra with involution $(A_1,
  \s_1)$ over $R_1$ such that $A \ox_R R_1$ and $A_1$ are Brauer equivalent over
  $S_1:=S\ox_R R_1$, $\s$ and $\s_1$ are both unitary, or are both non-unitary,
  and such that at least one of the following holds:
    \begin{enumerate}
      \item $Z(A_1)\cong R_1\x R_1$;
      \item $\deg A_1=1$;
      \item $\deg A_1$ is a $2$-power
        and there exists $u\in A_1^\x$ such 
        that $\s_1(u)=-u$
        (and other properties that are not needed in this proof).
    \end{enumerate}
  
  Observe that there is $\beta \in \Sperm R_1$ such that $\alpha = \beta\cap
  R$. Indeed, by Proposition~\ref{extend-nil}(2), there is $\gamma \in 
  \Sper R_1$
  such that $\alpha = \gamma \cap R$. Then, let $\beta \in \Sperm R_1$ be such
  that $\gamma \subseteq \beta$. It follows that 
  $\alpha = \beta\cap R$ since $\alpha$
  is maximal. Furthermore, $\beta\not\in \Nil[A\ox_R R_1, \s\ox\id]$ by
  Proposition~\ref{extend-nil}(3).
  \medskip

  We consider the three cases above:
  \medskip
    
  (1) This case does not occur: Since $Z(A \ox_R R_1) = Z(A_1) \cong R_1 \x R_1$,
  the involution $\s$ is necessarily unitary, and we obtain
  $\beta\in \Nil[A\ox_R R_1, \s\ox\id]$ by 
  \cite[Lemma~3.7]{A-U-Az-PLG},  contradiction.
  \medskip
  
  (2) In this case $A_1=S_1$.
    If $\s_1$ is unitary, then $S_1$ is a 
    quadratic \'etale $R_1$-algebra and $\s_1$ is the standard involution
    on $S_1$ by Proposition~\ref{rem:quad-et}. 
    \medskip
    
    \emph{Claim.} $K:=A_1\ox_{R_1} \kappa(\beta)$ is a quadratic field 
    extension of 
    $\kappa(\beta)$ of the form $\kappa(\beta)(\sqrt{d})$ with 
    $d<_{\bar\beta} 0$ and  $\iota:=\s_1\ox\id$ is 
    the canonical involution on $K$. 
    \smallskip

    \emph{Proof of the Claim.} Indeed, if $d>_{\bar\beta} 0$ or if $K\cong
    \kappa(\beta)\x \kappa(\beta)$, then $K\ox_{\kappa(\beta)} k(\beta)\cong
    k(\beta)\x k(\beta)$, and thus $\bar\beta\in \Nil[K,\iota]$ by
    \cite[Lemma~3.7]{A-U-Az-PLG}, i.e., $\bar\beta\in
    \Nil[A_1\ox_{R_1} \kappa(\beta), \s_1\ox\id]$. It follows that $\beta\in
    \Nil[A_1,\s_1]$, again by \cite[Lemma~3.7]{A-U-Az-PLG}. Since $\s_1$ is
    unitary, $\s$ is unitary by assumption, and it follows that $\s\ox\id$ is
    unitary by Proposition~\ref{prop:same_type}. Since $A\ox_R R_1$ and $A_1$
    are Brauer equivalent, it follows that
    $\Nil[A\ox_R R_1,\s\ox\id]=\Nil[A_1,\s_1]$, so that $\beta \in
    \Nil[A\ox_R R_1,\s\ox\id]$, and thus that $\alpha\in\Nil[A,\s]$ by
    Proposition~\ref{extend-nil}(3), contradiction. The second statement is
    clear since $\iota$ is $\kappa(\beta)$-linear. End of the proof of the
    Claim. 
    \medskip
    
    By the Claim and \cite[Remark~2.5]{A-U-sign-pc} we then obtain  
    $\sign^\mu_\beta \qf{1}_{\s_1}:=\sign^{\mu\ox\kappa(\beta)}_{\bar\beta} 
    \qf{1}_{\iota}=1$. Hence, $\sign^\eta_\beta 
    \qf{1}_{\s_1}=:\delta\in\{-1,1\}$. Taking $h_1=\qf{\delta}_{\s_1}$ then
    yields $\sign^\eta_\beta h_1=1$.
    \medskip

    If $\s_1$, and thus $\s$ is not unitary, we have $(A_1,
    \s_1) =(S_1, \iota) = (R_1, \id)$. Therefore $\s_1$ is orthogonal, and
    taking $h_1=\qf{\delta}_{\s_1}$ as above again yields 
    $\sign^\eta_\beta h_1=1$. We now show that $\s$ must also be orthogonal.
    We have $A_1 \ox_{R_1}
    k(\beta) \cong k(\beta)$. Since $A \ox_R R_1$ and $A_1$ are Brauer
    equivalent, $(A \ox_R R_1) \ox_{R_1} k(\beta)$ and $A_1 \ox_{R_1} k(\beta)$
    are Brauer equivalent (cf. \cite[Proposition~7.3.3]{ford17}), 
    and thus $(A \ox_R R_1) \ox_{R_1} k(\beta)$ is Brauer
    equivalent to $k(\beta)$. If $\s$ is symplectic, then $\beta
    \in \Nil[A \ox_R R_1, \s \ox \id]$ by \cite[Remark~3.6]{A-U-Az-cont}, and so
    $\alpha \in \Nil[A,\s]$ by Proposition~\ref{extend-nil}(3), contradiction.
\medskip
  
  (3) We have $S_1:= S \ox_R R_1 = Z(A) \ox_R R_1 = Z(A \ox_R R_1)$ by
  Proposition~\ref{change-of-base}. Therefore, by Lemma~\ref{twodefs}, $S_1$ is
  a quadratic \'etale extension of $R_1$. It follows that $S_1$ is connected,
  for if not, then $S_1 \cong R_1 \x R_1$ by \cite[Lemma~1.16]{first23}, and the
  same argument as in case (1) provides a contradiction. If $\s_1$ is symplectic
  then $\Int(u) \circ \s_1$ is orthogonal, and if $\s_1$ is orthogonal then
  $\Int(u) \circ \s_1$ is symplectic (this follows from
  \cite[Proposition~2.7]{BOI} since it can be checked at each $\fp \in \Spec R$,
  cf. Definition~\ref{def-type}). It follows that, up to replacing $\s_1$ by
  $\Int(u) \circ \s_1$, we can assume that $\s$ and $\s_1$ are of the same type.
  By Lemma~\ref{deg2pow}, there is a nonsingular hermitian form $h_1$ over
  $(A_1, \s_1)$ such that $\sign^{\lambda_1}_\beta h_1$ is (up to sign) a power
  of $2$, where $\lambda_1$ is a reference form for $(A_1,\s_1)$. \medskip

  Therefore, in both cases (2) and (3), $S_1$ is connected, $\s$ and $\s_1$ are
  of the same type, and there is a nonsingular hermitian form $h_1$ over $(A_1,
  \s_1)$ whose signature is (up to sign) a power of $2$. By
  Theorem~\ref{great-stuff}, there are forms $\lambda$ and $h$ over $(A \ox_R
  R_1, \s \ox \id)$ with $h$ nonsingular such that $\sign^\lambda_\beta h$ is a
  power of $2$. Since $\sign^{\eta \ox R_1}_\beta h = \pm \sign^\lambda_\beta h$
  we can assume (up to replacing $h$ by $-h$ if needed), that $\sign^{\eta \ox
  R_1}_\beta h=2^\ell$ for some $\ell\in \N\cup\{0\}$.

  By Lemma~\ref{separation}(1), the subset of $\Sperm R_1$ of all maximal
  extensions of $\alpha$ to $R_1$ is finite, say equal to $\{\beta_1, \ldots,
  \beta_m\}$, with $\beta_1=\beta$. By Lemma~\ref{separation}(2) there are $b_1,
  \ldots, b_k \in R_1^\x$ such that $\{\beta_1, \ldots, \beta_m\} \cap \mathring
  H (b_1, \ldots, b_k) = \{\beta_1\}$. Observe that $\sign_\gamma \Pf{b_1,
  \ldots, b_k}$ is equal to $2^k$ if $\gamma = \beta_1$ and equal to $0$ if
  $\gamma \in \{\beta_2, \ldots, \beta_m\}$, where $\Pf{b_1, \ldots, b_k}$
  denotes the Pfister form $\qf{1, b_1}\ox \cdots \ox \qf{1, b_k}$.

  The form $\Tr^*_{A \ox_R R_1 / A} (\Pf{b_1, \ldots, b_k} \ox h)$  
  over $(A,\s)$ is nonsingular 
  (since $\Tr_{A \ox_R R_1 / A}$ is an involution trace, cf.
  Section~\ref{inv-trace}), 
  and by the trace formula, Theorem~\ref{ktf}, we obtain
  \begin{equation}\label{ktf-power-2}
    \sign^\eta_\alpha \Tr^*_{A \ox_R R_1 / A} (\Pf{b_1, \ldots, b_k} \ox h) =
    \sum_{i=1}^r \sign^{\eta\ox R_1}_{\gamma_i} (\Pf{b_1, \ldots, b_k} \ox h),
  \end{equation}
  for some orderings $\gamma_1, \ldots, \gamma_r \in \Sper R_1$ such that 
  $\gamma_1^*, \ldots, \gamma_r^*$ are all different and $\Sperm R_1/\alpha=
  \{\gamma_1^*, \ldots, \gamma_r^*\}$. Thus $r=m$ and, up to some renumbering, we have 
  $\gamma_i \subseteq \beta_i$ for
  $i = 1, \ldots, m$.

  Recalling that $|\sign^{\eta\ox R_1}_\bullet (\Pf{b_1, \ldots, b_k}
  \ox h)|$ is continuous on $\Sper R_1$ for the Harrison topology by
  \cite[Corollary~4.3]{A-U-Az-cont}, we obtain by \cite[Lemma~6.5]{A-U-Az-cont}
  that, for $i=1. \ldots, m$,
  \begin{align*}
    |\sign^{\eta\ox R_1}_{\gamma_i} (\Pf{b_1, \ldots, b_k}
    \ox h)| &= 
    |\sign^{\eta\ox R_1}_{\beta_i} (\Pf{b_1, \ldots, b_k} \ox h)| \\
    &= \begin{cases}
      2^k |\sign^{\eta\ox R_i}_{\beta_i} h| & \text{if } i=1 \\
      0 & \text{otherwise}.
    \end{cases}
  \end{align*}
  Therefore, Equation~\eqref{ktf-power-2} yields
  \[\sign^\eta_\alpha \Tr^*_{A \ox_R R_1 / A} (\Pf{b_1, \ldots, b_k} \ox h) 
  = \pm 2^k
  \sign^{\eta\ox R_1}_{\beta_1} h = \pm 2^{k+\ell},\]
  which is a power of two, up to replacing $h$ by $-h$.
\end{proof}

\begin{lemma}\label{spermax}
  Let $R$ be semilocal connected. If $\alpha \in \Sper R\sm \Nil[A,\s]$, then
  $\alpha^* \in \Sperm R\sm \Nil[A,\s]$.
\end{lemma}

\begin{proof}
  By 
  Proposition~\ref{sl=great} there exists a nonsingular hermitian form $h$ over
  $(A,\s)$ such that $\sign^\eta_\alpha h \not=0$. By  \cite[Corollary~4.3]{A-U-Az-cont}, 
  $|\sign^\eta_\bullet h|$
  is continuous for the Harrison topology. It then follows from  
  \cite[Lemma~6.5]{A-U-Az-cont}
  that $\sign^\eta_{\alpha^*} h \not=0$, so that $\alpha^*\not\in \Nil[A,\s]$.  
\end{proof}

\begin{thm}\label{tired} 
  Let $R$ be semilocal.
  There exists $m\in\N\cup\{0\}$ and a nonsingular  hermitian form $h_0$ 
  over 
  $(A,\s)$ such that $|\sign^\eta_\bullet h_0|=2^m$ 
  on $\Sper R \setminus \Nil[A,\s]$.
\end{thm}

\begin{proof}
  We assume that $\Sper R \setminus \Nil[A,\s] \not = \emptyset$, otherwise the
  result is trivially true. 

  Since $R$ is semilocal, we can write $R \cong R_1 \x \cdots \x R_t$ (for some
  $t \in \N$) where each $R_i$ is a connected semilocal ring (cf.
  \cite[Remark~6.5]{A-U-Az-PLG}). Let $e_i := (0, \ldots, 0,1,0, \ldots, 0)$
  (with $1$ in position $i$).  Since $e_i \in R$, we have $\s(e_i) = e_i$ for
  $i=1, \ldots, t$.  Writing $A_i := Ae_i$, it follows that
  \[(A,\s) = (A_1, \s|_{A_1}) \x \cdots \x (A_t, \s|_{A_t}).\]
  where $(A_i, \s_i) \cong (A \ox_R R_i, \s \ox \id)$ for $i=1, \ldots, t$, so
  is an Azumaya algebra with involution over $R_i$.
 
  Up to some renumbering, we can assume that $\Sper R_i \setminus \Nil[A,\s] \not
  = \emptyset$ for $1 \le i \le s$ and $\Sper R_i \setminus \Nil[A,\s] =
  \emptyset$ for $s+1 \le i \le t$.  Therefore, with reference to 
  Equation~\eqref{union}, we have   
  \[\Sper R \setminus \Nil[A,\s] = (\Sper R_1 \setminus \Nil[A_1, \s_1])
  \dot\cup \cdots \dot\cup (\Sper R_{s} \setminus \Nil[A_{s}, \s_{s}]).\]
  Let $i \in \{1, \ldots, s\}$.
  By Lemma~\ref{spermax} there exists some 
  $\alpha_i \in \Sperm R_i \setminus \Nil[A_i,\s_i]$, and by 
  Proposition~\ref{sl=great} there
  is a nonsingular hermitian form $(M_i,h_i)$ over $(A_i,\s_i)$ such that
  $\sign^{\eta \ox R_i}_{\alpha_i} h_i = 2^{k_i}$ for some $k_i \in \N$. Since $\Sper R_i$
  is connected and $|\sign^\eta_\bullet h_i|$ is continuous on $\Sper R_i$ 
  (cf. \cite[Corollary~4.3]{A-U-Az-cont}) we have
  $|\sign^\eta_\bullet h_i| = 2^{k_i}$ on $\Sper R_i$.

  For $i=s+1, \ldots, t$, we take for $(M_i,h_i)$ any nonsingular hermitian form
  over $(A_1, \s_1)$, and we define (cf. Definition~\ref{direct-product}):
  \[(M,h) := (M_1,h_1) \x \cdots \x (M_t,h_t).\]
  By Proposition~\ref{dp}, $h$ is nonsingular, and a direct verification shows
  that, if $\alpha \in \Sper R_i
  \setminus \Nil[A_i,\s_i]$ for some $i \in \{1,\ldots,s\}$, then we have
  \[|\sign^\eta_\alpha h| = |\sign^{\eta \ox R_i}_\alpha h\ox R_i| = |\sign^{\eta \ox
  R_i}_\alpha h_i| = 2^{k_i}.\]

  Let $k = k_1 + \cdots + k_s$. Since $\Sper R_i$ is clopen in $\Sper R$, 
  there exist $\ell_i \in \N$ and a nonsingular quadratic form $q'_i$ over
  $R$ such that $\sign q'_i = 2^{\ell_i+k-k_i}$ on $\Sper R_i$ and $\sign q'_i = 0$ on
  $\Sper R \setminus \Sper R_i$, cf.
  \cite[Theorem~3.2]{mahe82}. Let $\ell := \ell_1 + \cdots + \ell_s$ and $q_i :=
  2^{\ell-\ell_i} \x q'_i$ for $i = 1, \ldots, s$ (so that $\sign q_i =
  2^{\ell+k-k_i}$ on $\Sper R_i$ and $0$ on its complement). Then 
  \[h_0 := q_1 \ox h \perp \ldots \perp q_s
  \ox h\] 
  satisfies $|\sign^\eta h_0| = 2^{\ell + k}$ on $\Sper R$.
\end{proof}

\section{An exact sequence for total
signatures in the semilocal case}\label{sec5}

In this section, we assume that the ring $R$ is semilocal (recall our standing
hypothesis that $2 \in R^\x$), and also that the 
reference form $\eta$ for $(A,\s)$ is nonsingular, which is justified 
by Theorem~\ref{tired}.

We will briefly consider the link between continuous functions
defined on $\Sper R$ and total signatures of hermitian forms. The arguments are
the same, \emph{mutatis mutandis} as those in \cite{A-U-stab}. Since any
hermitian form over $(A,\s)$ has signature $0$ on $\Nil[A,\s]$, we introduce the
notation:
\begin{align*}
  C(\Sper R,\Z)_{[A,\s]} := \{f:\Sper R \rightarrow \Z \mid f &\text{ is
    continuous for the Harrison} \\
    &\text{topology and } f=0 \text{ on } \Nil[A,\s]\}.
\end{align*}

\begin{prop}\label{2-primary}
  Let $f \in C(\Sper R,\Z)_{[A,\s]}$. Then there exist $m \in \N$
  and a nonsingular hermitian form $h$ over $(A,\s)$ such that $2^m f =
  \sign^\eta_\bullet h$.
\end{prop}
\begin{proof}
  Let $h_0$ be a nonsingular hermitian form over $(A,\s)$ such that 
  $|\sign^\eta_\bullet
  h_0| = 2^k$ on $\Sper R \setminus \Nil[A,\s]$ (cf. Theorem~\ref{tired}).
  Consider the map $\delta : \Sper R \setminus \Nil[A,\s] \rightarrow \Z$, $\delta(\alpha) :=
  \sgn(\sign^\eta_\alpha h_0)$, so that $\delta(\alpha) \sign^\eta_\alpha h_0 =
  2^k$. By \cite[Theorem~6.1]{A-U-Az-cont} $\delta$ is continuous for the
  Harrison topology. Let $q$ be a nonsingular quadratic form over $R$ and $\ell \in \N$ such
  that $\sign_\bullet q = 2^\ell \delta f$ on $\Sper R$
  (cf. \cite[Th\'eor\`eme~3.2]{mahe82}). Then
  $q\ox h_0$ is nonsingular and
  $\sign^\eta_\bullet q \ox h_0 = 2^{\ell+k} f$.
\end{proof}

\begin{defn}
  We define $S^\eta(A,\s)$ to be the cokernel of the map $\sign^\eta_\bullet : W(A,\s)
  \rightarrow C(\Sper R, \Z)_{[A,\s]}$.
\end{defn}

Denoting by $W_t(A,\s)$ the torsion subgroup of $W(A,\s)$, we have:
\begin{thm}
  The sequence
  \[0 \rightarrow W_t(A,\s) \rightarrow W(A,\s)\xrightarrow{\sign^\eta_\bullet} 
  C(\Sper R, \Z)_{[A,\s]} \rightarrow
  S^\eta(A,\s) \rightarrow 0\]
  is exact, and the groups $W_t(A,\s)$ and $S^\eta(A,\s)$ are 2-primary torsion.
\end{thm}
\begin{proof}
  This follows from Pfister's local-global principle, 
  \cite[Theorem~6.6]{A-U-Az-PLG}, and
  Proposition~\ref{2-primary}.
\end{proof}

\begin{rem}
  If $(A,\s)=(F,\id)$ with $F$ a formally real field, then $S^\eta(A,\s)$ 
  is equal to the stability
  group $S(F)$ of $F$, 
  and the element $s\in \N\cup\{\infty\}$ such that the exponent
  of $S(F)$ is $2^s$ is called the stability index of $F$, cf. 
  \cite[\S1]{br74}.  
\end{rem}

\appendix

\section{Hermitian Morita equivalence}

In this appendix we show that Brauer equivalent Azumaya algebras with
involution over semilocal connected rings are hermitian Morita equivalent,
and that this equivalence is preserved under semilocal connected
base change. For the convenience
of the reader we spell out most of the details. 

Recall that $(A,\s)$ denotes an Azumaya algebra with involution over $R$. 
Let $\ve\in Z(A)$ be a unit such that $\s(\ve)\ve=1$. We denote the category
of $\ve$-hermitian forms over $(A,\s)$ by $\Herm^\ve(A,\s)$. Its objects are
the $\ve$-hermitian modules $(M,h)$ over $(A,\s)$ and its morphisms are the
isometries (in particular they are all isomorphisms).

Let
$(B,\tau)$ be a further Azumaya algebra with involution over $R$, and  let $\delta\in Z(B)$ be a unit such that $\tau(\delta)\delta=1$.
We recall from \cite[Chapter~I, \S8]{knus91} how the product of forms is defined. 
Let $(N,\vf)$ be an $\ve$-hermitian form over $(A,\s)$ and let $(M,\psi)$ be 
a $\delta$-hermitian form over $(B,\tau)$.  Assume that $N$ is a $B$-$A$-bimodule over $R$ (i.e., the left and right actions of $R$ on $N$
coincide) and assume that $\vf$ \emph{admits} $(B,\tau)$, i.e., 
\[
\vf(n\star b, n')=\vf(n, b n') 
\]
for all $n,n'\in N$ and all $b\in B$, where the right action $\star$ of $B$ on $N$
is via the involution $\tau$, i.e., $n\star b = \tau(b)n$.
The \emph{product of $\vf$ and $\psi$} is then the
$\ve\delta$-hermitian form
$\vf\bullet\psi: (M\ox_B N) \x (M\ox_B N) \to A$ over $(A,\s)$
defined by
\begin{equation}\label{eqbullet}
  \vf\bullet\psi (m\ox n, m'\ox n'):= \vf(n, \psi(m,m')n').
\end{equation}
The product is associative and distributive with respect to orthogonal sums in
either slot. It is nonsingular if and only if both
$\vf$ and $\psi$ are nonsingular. 
\medskip

If $(M,h)$ is a nonsingular $\ve$-hermitian form over
$(A,\s)$, the \emph{adjoint involution of $h$} is the
involution $\ad_h$ on the ring $\End_A(M)$ implicitly defined by
\[
  h(x, \ad_h(f)(y))= h(f(x), y)
\]
for all $x,y\in M$ and all $f\in \End_A(M)$, cf. \cite[Section~2.4]{first23}.

\begin{thm}\label{morprop1}
  Let $R$ be semilocal connected and let $(B,\tau)$ be an Azumaya algebra
  with involution over $R$ such that $A$ and $B$ are Brauer equivalent and
  $\s|_S=\tau|_S$, where $S:=Z(A)=Z(B)$. Assume that $S$ is connected.
  Then there exists $\delta\in\{-1,1\}$ and a nonsingular $\delta$-hermitian form
  $(P,\vt)$ over $(B,\tau)$ which admits $(A,\s)$ such that the functor
  \begin{align*}
    \CF_{\vt}:\Herm^\ve(A,\s) &\to \Herm^{\delta\ve}(B,\tau) \\
    (M,h)&\mapsto (M\ox_A P, \vt \bullet h) & 
    \text{(objects)}\\
    f &\mapsto f\ox \id_{P} & \text{(morphisms)}
  \end{align*}
  is an equivalence of categories that preserves orthogonal sums, maps
  nonsingular forms to nonsingular forms, and hyperbolic forms to hyperbolic
  forms. Specifically, if $\s|_S=\id_S$, then
  $\delta=1$ if $\s$ and $\tau$ are both orthogonal or both symplectic and
  $\delta=-1$ otherwise; if $\s|_S\not=\id_S$, then $\delta$ can be chosen
  freely in $\{-1,1\}$. 
\end{thm}

\begin{proof}
  This is an extended version of the proof of \cite[Theorem~2.15]{A-U-Az-PLG}.

  If $\s|_S\not=\id_S$,  then $S$ is a quadratic \'etale
  $R$-algebra by \cite[Proposition~2.12]{A-U-Az-PLG}
  and is thus semilocal by \cite[Chapter~VI, Proposition~1.1.1]{knus91}. 
  
  Since $A$ and $B$ are Brauer equivalent, there exist faithfully projective
  (i.e., faithful
  finitely generated 
  projective) $S$-modules $P$ and $Q$ such that $A\ox_S \End_S(P)\cong
  B\ox_S \End_S(Q)$, cf. \cite[Chapter~III, (5.3)]{knus91}. Since $S$
  is semilocal connected, $P$ and $Q$ are free  
  by \cite{hinohara}. It 
  follows that
  there exist $k,\ell \in \N$ and an isomorphism of Azumaya $S$-algebras
  \begin{equation}\label{end-iso}
   \zeta: \End_A(A^k)\to \End_B(B^\ell). 
  \end{equation}

  Let us write $\tilde A=\End_A(A^k)$, $\tilde B=\End_B(B^\ell)$, and
  note that  $Z(\tilde A)=Z(\tilde B)=S$.
  Consider the nonsingular hermitian forms $(A^k, \vf:=k\x \qf{1}_\s)$ and
  $(B^\ell, \psi:=\ell\x \qf{1}_\tau)$ over $(A,\s)$ and $(B,\tau)$, 
  respectively.   
  Since $\vf$ and $\psi$ are hermitian, the involutions $\s$ and $\ad_\vf$
  are of the same type and the same is true for the involutions $\tau$
  and $\ad_\psi$, cf. \cite[Proposition~2.11]{first23}. In particular,
  $\s|_S=\ad_\vf|_S$ and $\tau|_S=\ad_\psi|_S$. 
  Note that $A^k$  is naturally an $\tilde A$-$A$-bimodule over $R$
  and $B^\ell$ is naturally a $\tilde B$-$B$-bimodule over $R$.

  Since $A^k$ is faithfully projective over $A$ (since it is free) and $\vf$
  is nonsingular, \cite[Chapter~I, Proposition~9.3.4]{knus91} tells us that $(A^k, \vf)$
  admits $(\tilde A, \ad_\vf)$ and furnishes us with a nonsingular hermitian
  form $(\ovl{A^k}, \vf')$ over $(\tilde A, \ad_\vf)$ which admits $(A,\s)$.
  Here, $\vf'$ is constructed from the adjoint linear map of $\vf$ (the 
  details do not matter for our argument), and  
  $\ovl{A^k}$ is the $A$-$\tilde A$-bimodule obtained from $A^k$ by
  twisting the multiplication with the involutions $\s$ and $\ad_\vf$, 
  respectively: 
  \[
    a\cdot t\cdot b:=\ad_\vf(b)t\s(a) \ (= \ad_\vf(b)(t\s(a))  )
  \]
  for all $a\in A$, $t\in A^k$ and $b\in\tilde A$.
   Then 
  \cite[Chapter~I, Theorem~9.3.5]{knus91} and its proof
  show that the functor 
  \begin{align*}
    \CG:\Herm^\ve(A,\s) &\to \Herm^{\ve}(\tilde A, \ad_\vf)\\
    (M,h)&\mapsto (M\ox_A \ovl{A^k}, \vf'\bullet h)\\
    f &\mapsto f\ox \id_{\ovl{A^k}},
  \end{align*}
  is an equivalence of categories.

  The involution $\ad_\vf$ on $\tilde A$ induces an involution $\omega$
   of the same type on $\tilde B$ via the isomorphism $\zeta$. Thus we have an
  isomorphism $(\tilde A,\ad_\vf)\cong (\tilde B, \omega)$ of Azumaya algebras
  with involution over $S$. By \cite[Chapter~I, (7.1)]{knus91} there is an
  isomorphism of categories
  \begin{align*}
    \CU:\Herm^\ve(\tilde A,\ad_\vf) &\to \Herm^{\ve}
        (\tilde B, \omega)\\
    (M,h)&\mapsto (M\ox_{\tilde A} \tilde B, \zeta^*h)\\
    f &\mapsto f\ox \id_{\tilde B},
  \end{align*}
  where 
  \[\zeta^*h (m\ox a, n\ox b):= \omega(a) \zeta(h(m,n))b\] 
  for all 
  $m,n\in M$ and all $a,b\in \tilde B$.

  Observe that by our assumptions, $\omega|_S =\ad_\psi|_S$. By the
  Skolem-Noether theorem \cite[Theorem~8.6]{BFP}, $\omega$ then differs from
  $\ad_\psi$ by an inner automorphism: there exists $\delta \in \{-1,1\}$ and a
  unit $u\in \tilde B$ with $\ad_\psi(u)=\delta u$ such that $\omega =
  \Int(u)\circ \ad_\psi$, where $\delta$ can be freely chosen in $\{-1,1\}$ if
  $\s|_S \not = \id_S$. Moreover, $\delta$ is unique if $\s|_S=\id_S
  (=\ad_\psi|_S)$: if $\Int(u)=\Int(u')$, then $u=u's$ for some $s\in S^\x$,
  and it follows that $\ad_\psi(u)=\delta u$ if and only if
  $\ad_\psi(u')=\delta u'$. Then \cite[Chapter~I, Remark~5.8.2]{knus91}
  yields an isomorphism of categories
  \begin{align*}
    \CS: \Herm^\ve(\tilde B,\omega) &\to \Herm^{\delta\ve} 
    (\tilde B, \ad_\psi)\\
    (M,h)&\mapsto (M,u^{-1}h)\\
    f&\mapsto f. 
  \end{align*}

  Since $B^\ell$ is faithfully projective over $B$ (since it is free) and
  $\psi$ is nonsingular, \cite[Chapter~I, Proposition~9.3.4]{knus91} tells us 
  that
  $(B^\ell, \psi)$ admits $(\tilde B, \ad_\psi)$, and  
  \cite[Chapter~I, Theorem~9.3.5]{knus91} 
  shows that the functor 
  \begin{align*}
    \CF:\Herm^{\delta\ve}(\tilde B,\ad_\psi) &\to \Herm^{\delta\ve}(B, 
    \tau)\\
    (M,h)&\mapsto (M\ox_{\tilde B} B^\ell, \psi\bullet h)\\
    f &\mapsto f\ox \id_{B^\ell},
  \end{align*}
  is an equivalence of categories.

  Putting it all together, we end up with the following  
  sequence of categories and functors:
  \[
    \Herm^\ve(A,\s) \xrightarrow{\CG} \Herm^\ve (\tilde A, \ad_\vf)
    \xrightarrow{\CU} \Herm^\ve(\tilde B, \omega) \xrightarrow{\CS}
    \Herm^{\delta\ve}(\tilde B, \ad_\psi) \xrightarrow{\CF} 
    \Herm^{\delta\ve}(B,\tau).
  \]
  The composition $\CE:=\CF\circ\CS\circ\CU\circ\CG$ is an equivalence
  of categories since the constituent functors are equivalences and isomorphisms
  of categories. Furthermore, $\CE$
  preserves orthogonal sums, maps nonsingular forms to nonsingular forms, and
  hyperbolic forms to hyperbolic forms since the constituent functors have
  these properties by the references in \cite{knus91} already cited.

  Let
    \[
      (P,\vt):= \bigl( (\ovl{A^k}\ox_{\tilde A}\tilde B)
        \ox_{\tilde B}B^\ell, \psi \bullet u^{-1}\zeta^* \vf'  \bigr), 
    \]
    and let $\CF_\vt$ be as defined in the statement of the theorem. Then
    the functors $\CE$ and $\CF_\vt$ are naturally isomorphic by 
    Lemma~\ref{lemmaA2} 
    below. We conclude that $\CF_\vt$ is an equivalence of categories that
    preserves orthogonal sums, maps nonsingular forms to nonsingular forms, and
    hyperbolic forms to hyperbolic forms.

  Finally we show the claim about $\delta$ when $\s|_S=\id_S$. Since $\delta$
  is unique, it suffices to check its value at any $\fp \in \Spec R$ (i.e.,
  after tensoring over $R$ by $\kappa(\fp)$). By \cite[Proposition~2.7]{BOI} we
  know that $\delta = 1$ if $\s$ and $\tau$ are both orthogonal or both
  symplectic at $\fp$ and $\delta = -1$ otherwise. The statement then follows,
  using \cite[Proposition~2.12]{A-U-Az-PLG}.
\end{proof}

\begin{lemma}\label{lemmaA2}
  The functors $\CE$ and $\CF_\vt$ are naturally isomorphic.
\end{lemma}

\begin{proof}
  Let $(M,h)\in \Herm^\ve(A,\s)$. We first compute $\CE(M,h)$:
  \begin{align*}
    \CE(M,h)&=(\CF\circ\CS\circ\CU\circ\CG)(M,h)\\
      &= (\CF\circ\CS\circ\CU)(M\ox_A \ovl{A^k}, \vf'\bullet h)\\
      &=(\CF\circ\CS)\bigl( (M\ox_A \ovl{A^k})\ox_{\tilde A}\tilde B, 
          \zeta^*(\vf'\bullet h)  \bigr)\\
      &=\CF\bigl( (M\ox_A \ovl{A^k})\ox_{\tilde A}\tilde B, 
          u^{-1}(\zeta^*(\vf'\bullet h))  \bigr)\\
      &=\Bigl( \bigl((M\ox_A \ovl{A^k})\ox_{\tilde A}\tilde B\bigr)
      \ox_{\tilde B} B^\ell, 
        \psi\bullet ( u^{-1}(\zeta^*(\vf'\bullet h)))  \Bigr).
  \end{align*}
  Next we show that there is an isometry
  \[\CE(M,h)\xrightarrow{\alpha_{(M,h)}} \CF_\vt(M,h).\]
  Indeed:
  \begin{align*}
    \CE(M,h)      &=\Bigl( \bigl((M\ox_A \ovl{A^k})\ox_{\tilde A}\tilde B\bigr)
      \ox_{\tilde B} B^\ell, 
        \psi\bullet ( u^{-1}(\zeta^*(\vf'\bullet h)))  \Bigr)\\ 
      &\simeq_1 \Bigl( \bigl(M\ox_A (\ovl{A^k}\ox_{\tilde A}\tilde B)\bigr)\ox_{\tilde B} 
                B^\ell, \psi\bullet ( u^{-1}(\zeta^*\vf'\bullet h))  \Bigr)\\
      &= \Bigl( \bigl(M\ox_A (\ovl{A^k}\ox_{\tilde A}\tilde B)\bigr)\ox_{\tilde B} 
                B^\ell, \psi\bullet ( u^{-1}\zeta^*\vf'\bullet h)  \Bigr)\\     
      &\simeq_2  \Bigl( M\ox_A \bigl( (\ovl{A^k}\ox_{\tilde A}\tilde B)
      \ox_{\tilde B}B^\ell\bigr), 
          (\psi \bullet u^{-1}\zeta^* \vf')\bullet h  \Bigr)\\
      &= \CF_\vt (M,h),
  \end{align*}
  where $\simeq_1$ and $\simeq_2$ are isometries. 
  In order to verify $\simeq_1$, it suffices to show that
  \[
    \bigl( (M\ox_A \ovl{A^k})\ox_{\tilde A}\tilde B, 
            \zeta^*(\vf'\bullet h)  \bigr)
    \simeq 
    \bigl( M\ox_A (\ovl{A^k}\ox_{\tilde A}\tilde B), 
    \zeta^*\vf'\bullet h  \bigr).
  \]
  By associativity of the tensor product there is a canonical
  isomorphism induced by
  \[
  \nu: (M\ox_A \ovl{A^k})\ox_{\tilde A}\tilde B \to
  M\ox_A (\ovl{A^k}\ox_{\tilde A}\tilde B),\ (m\ox t)\ox b\mapsto m\ox (t\ox b)
  \]
  for all $m\in M$, $t\in \ovl{A^k}$ and $b\in \tilde B$. We check that it is
  an isometry:
  \begin{align*}
    \zeta^*\vf'\bullet h \bigl(\nu((m_1\ox t_1)\ox b_1), &\ \nu((m_2\ox t_2)\ox 
    b_2)      \bigr)\\
    &=  \zeta^*\vf'\bullet h \bigl(m_1\ox (t_1\ox b_1), m_2\ox (t_2\ox 
    b_2)      \bigr)   \\
    &=\zeta^* \vf' \bigl(t_1\ox b_1, h(m_1,m_2) (t_2\ox b_2)\bigr)\\
    &= \omega(b_1) \zeta \bigl(\vf'(t_1, h(m_1,m_2)t_2\bigr)b_2\\
    &= \omega(b_1) \zeta \bigl( \vf'\bullet h (m_1\ox t_1, m_2\ox 
    t_2)\bigr)b_2\\
    &= \zeta^* (\vf'\bullet h) \bigl( (m_1\ox t_1)\ox b_1, (m_2\ox t_2)\ox 
    b_2      \bigr).
  \end{align*}
  The verification of $\simeq_2$ is similar (or, see  \cite[Chapter~I, 
  Lemma~8.1.1(2)]{knus91}).
    
  Now consider an isometry $(M,h)\xrightarrow{f}(N,h')$. 
  A straightforward computation gives
  $\CE(f)=((f\ox\id_{\ovl{A^k}})\ox\id_{\tilde B})\ox \id_{B^\ell}$, while
  $\CF_\vt(f)=f\ox\id_{(\ovl{A^k}\ox\tilde B)\ox B^\ell}$ by definition.
    
  Finally, the commutativity of the square of isomorphisms
  \[
  \begin{split}
    \xymatrix{
      \bigl((M\ox_A \ovl{A^k})\ox_{\tilde A}\tilde B\bigr)
            \ox_{\tilde B} B^\ell \ar[r]^{\CE(f)} \ar[d]_{\alpha_{(M,h)}}
            & \bigl((N\ox_A \ovl{A^k})\ox_{\tilde 
            A}\tilde B\bigr) \ox_{\tilde B} B^\ell \ar[d]^{\alpha_{(N,h')}}\\
      M\ox_A \bigl( (\ovl{A^k}\ox_{\tilde A}\tilde B)
            \ox_{\tilde B}B^\ell\bigr) \ar[r]_{\CF_\vt(f)}
            &
            N\ox_A \bigl( (\ovl{A^k}\ox_{\tilde A}\tilde B)
            \ox_{\tilde B}B^\ell\bigr)     
        }
   \end{split} ,     
  \]
  determined by
  \[
  \begin{split}
    \xymatrix{
    ((m\ox t)\ox b)\ox s \ar@{|->}[r]^{\CE(f)} \ar@{|->}[d]_{\alpha_{(M,h)}} & 
    ((f(m)\ox t)\ox b)\ox s \ar@{|->}[d]^{\alpha_{(N,h')}} \\
    m\ox ((t\ox b)\ox s) \ar@{|->}[r]_{\CF_\vt(f)} & f(m)\ox ((t\ox b)\ox s)
    }
    \end{split} ,
  \]
  yields the commutative square
  \[
   \begin{split}
    \xymatrix{
    \CE(M,h) \ar[r]^--{\CE(f)} \ar[d]_{\alpha_{(M,h)}} & \CE(N,h') 
    \ar[d]^{\alpha_{(N,h')}}\\
    \CF_\vt(M,h) \ar[r]_--{\CF_\vt(f)} & \CF_\vt (N,h')
    }
    \end{split} .
  \]
  Hence, there is a natural transformation $\alpha: \CE\to \CF_\vt$, which is a 
  natural isomorphism since the isometries $\alpha_{(M,h)}$ are isomorphisms.
\end{proof}

Let $T$ be a commutative $R$-algebra. 
Let $h:M\x M\to A$ be an
$\ve$-hermitian form 
over $(A,\s)$. Then $(M,h)$ yields the $\ve\ox 1$-hermitian module 
$(M_T, h_T):=(M\ox_R T, h\ox T)$ over $(A\ox_R T, \s\ox\id_T)$,
where 
\begin{equation}\label{eqht}
h\ox T (m\ox t, m'\ox t') := h(m,m')\ox tt'
\end{equation}
for all $m,m'\in M$ and all
$t,t'\in T$. If $(M,h)$ is nonsingular or hyperbolic, then $(M_T,h_T)$ is
also nonsingular or hyperbolic, respectively, cf. \cite[p.~38]{knus91}.

Let $(M,h) \xrightarrow{f} (N,h')$ be an isometry of 
$\ve$-hermitian modules over $(A,\s)$,
then $ (M_T, h_T) \xrightarrow{f\ox \id_T} (N_T, h'_T)$ is the corresponding isometry of $\ve\ox 1$-hermitian
modules over $(A\ox_R T, \s\ox\id_T)$. Putting it all together yields a functor
\[
  \CU_{(A,\s,T)}: \Herm^\ve (A,\s)\to \Herm^{\ve\ox 1} (A\ox_R T, \s\ox\id_T)
\]
which
preserves orthogonal sums, maps nonsingular forms to 
nonsingular forms,    and hyperbolic forms to hyperbolic forms,
cf.  \cite[p.~38]{knus91}.

\begin{thm}\label{Brcom}
  Let $R$ be semilocal connected and let $(B,\tau)$ be an Azumaya algebra
  with involution over $R$ such that $A$ and $B$ are Brauer equivalent and
  $\s|_S=\tau|_S$, where $S:=Z(A)=Z(B)$. Assume that $S$ is connected.
  Let $R\to T$ be a homomorphism of commutative rings such that $T$ is semilocal
  connected and  $S_1:=S\ox_R T$ is connected.  
  Then $A\ox_R T$
  and $B\ox_R T$ are Brauer equivalent over $S_1$
  and the diagram of categories and functors
  \begin{equation}\label{diagfunct}
  \begin{split}
  \xymatrix{
  \Herm^{\ve\ox 1}(A\ox_R T, \s\ox\id)\ar[r]^--{\CF_{\vt_T}} & 
      \Herm^{\delta\ve\ox 1}(B\ox_R T, 
  \tau\ox\id)\\
  \Herm^\ve(A,\s) \ar[r]_--{\CF_{\vt}} \ar[u]^{\CU_{(A,\s,T)}} & 
  \Herm^{\delta\ve}(B,\tau) \ar[u]_{\CU_{(B,\tau,T)}}
  }
  \end{split}
  \end{equation}
  commutes (up to a natural isomorphism). 
\end{thm}

\begin{proof}
  The ring homomorphism $R\to T$ yields a ring homomorphism $S\to S_1$. It
  follows from \cite[Proposition~7.3.3]{ford17} that $A\ox_S S_1$ and $B\ox_S
  S_1$ are Brauer equivalent over $S_1$. Since $A\ox_R T \cong A\ox_S (S\ox_R
  T)=A\ox_S S_1$ and similarly for $B$, the first claim follows.

 It follows from Proposition~\ref{prop:same_type} that $\s$ and $\s\ox\id$
  are of the same type, and also that $\tau$ and $\tau\ox\id$ are of the same
  type. Furthermore, $(\s\ox\id)|_{S_1}=(\tau\ox\id)|_{S_1}$.
  Theorem~\ref{morprop1} then yields the rows of Diagram~\eqref{diagfunct}.
  
  Consider $(M,h)\in \Herm^\ve(A,\s)$. Then
  \begin{align*}
    (\CF_{\vt_T}\circ \CU_{(A,\s,T)}) (M,h) &= 
        \bigl( M_T\ox_{A\ox_R T} P_T, \vt_T\bullet  h_T\bigr)\\
      &\simeq \bigl((M\ox_A P)_T, (\vt\bullet h)_T 
       \bigr)\\
      &= (\CU_{(B,\tau,T)} \circ \CF_\vt)(M,h), 
  \end{align*}
  where the isometry $\simeq$ follows via a straightforward computation
  from the isomorphism of right 
  $T$-modules in \cite[Proof of Lemma~4.12]{A-U-Az-PLG}, induced by
  \[
    M_T \ox_{A\ox_R T} P_T \to (M\ox_A P)_T,\ (m\ox t_1)\ox (p\ox t_2)
          \mapsto (m\ox p)\ox t_1 t_2.
  \]        
  This isometry induces a natural isomorphism 
  $\CF_{\vt_T}\circ \CU_{(A,\s,T)} \cong \CU_{(B,\tau,T)} \circ \CF_\vt $,
  which concludes the proof.
\end{proof}


\end{document}